\documentclass[english]{amsart}	
\pdfoutput=1
\usepackage{amsmath}

\usepackage{amssymb}

\usepackage{amsthm}

\usepackage{epsfig}

\usepackage[T1]{fontenc}
\usepackage[latin9]{inputenc}
\usepackage{geometry}
\usepackage{verbatim}
\usepackage{float}
\usepackage{units}
\usepackage{textcomp}
\usepackage{mathrsfs}
\usepackage{amsthm}
\usepackage{amsmath}
\usepackage{amssymb}
\usepackage{graphicx}

\usepackage{pdflscape}

\usepackage{babel}

\newtheorem{theorem}{Theorem}

\theoremstyle{plain}
\newtheorem{lemma}{Lemma}
\newtheorem{proposition}{Proposition}
\newtheorem{corollary}{Corollary}

\newtheorem{question}{Question}

\theoremstyle{definition}
\newtheorem{definition}{Definition}

\theoremstyle{remark}
\newtheorem{remark}{Remark}

\begin{document}

\title{Polynomial Time Computable Triangular Arrays for Almost Sure Convergence}

\date{8 March, 2016}

\author{Vladimir \mbox{Dobri\ensuremath{\acute{\text{c}}^{\dagger{ }}}} }
\address{Department of Mathematics, Lehigh University, 14 East Packer Avenue, 
Bethlehem, PA 18015}
\email{vd00@lehigh.edu}

\author{Marina Skyers}
\address{Department of Mathematics, Lehigh University, 14 East Packer Avenue, 
Bethlehem, PA 18015}
\email{marinaskyers@gmail.com}

\author{Lee J. Stanley}
\address{Department of Mathematics, Lehigh University, 14 East Packer Avenue, 
Bethlehem, PA 18015}
\email[Corresponding author]{ljs4@lehigh.edu}

\thanks{Skyers and Stanley dedicate this paper to the memory of our dear departed friend and co-author, 
Vladimir \mbox{Dobri\ensuremath{\acute{\text{c}}}}.
The important contribution of S. Buss
will be explicitly acknowledged at various points in the body of the paper.  
We would also like to thank P.
Clote and A. Nerode for helpful discussions of complexity issues, and to
thank Nerode for pushing us to improve our upper complexity bounds and
suggesting the connection between tameness and continuity.  We are indebted
to our Lehigh colleagues, Vince Coll, Daniel Conus, Rob Neel and Joe Yukich for many helpful
comments and suggestions.}



\begin{abstract}
For $x \in (0,1)$, write
$x = \sum_{i=1}^\infty \varepsilon_i(x)2^{-i}$, with each $\varepsilon_i(x) \in \{ 0,1\}$ and
$\varepsilon_i(x) = 0$ for infinitely many $i$.  Let $R_i(x) := 
(-1)^{1 + \varepsilon_i(x)}$ and $\left\{ S_n\right\}$ be the random walk on 
$\mathbb{Z}$ defined on $(0,1):
S_n =\sum_{i=1}^nR_i$.   
By the Central Limit Theorem, the sequence $\left\{S_n/\sqrt{n}\right\}$
converges weakly to the standard normal distribution on $(0,1)$.
It is well known that there are $S^*_n$, equal
in distribution to $S_n$, for which Skorokhod showed that $\left\{S^*_n/\sqrt{n}\right\}$ converges
almost surely to the standard normal on $(0,1)$.

We introduce a general method for constructing from $\left\{R_i\right\}$ triangular array representations
$\left ( R^*_{n,i}\vert 1 \leq i \leq n,\ n \in \mathbb{Z^+}\right )$ of $\left\{ S^*_n\right\}$, where
each $R^*_{n,i}$ is a mean 0, variance 1 Rademacher random variable depending only on the first $n$ bits
of the binary expansion of $x \in (0,1)$.  These representations are {\it strong} in that for each
$n,\ S^*_n$ is equal to the sum of the i.i.d family, $\left ( R^*_{n,i}\vert 1 \leq i \leq n\right )$, pointwise,
as a function on $(0,1)$, not just in distribution.  Our construction method gives a bijection between
the set of all representations with these properties and the set of sequences, $\left\{ \pi_n\right\}$, 
where each $\pi_n$  is a permutation of $\{ 0, \ldots 2^n - 1\}$ with the property that  we call {\it admissibility}.  

We show that the complexity of any sequence of admissible permutations is bounded below by
the complexity of $2^n$, the exponential  function on natural numbers with base 2.  We explicitly construct
three such sequences which are polynomial time computable and whose complexity is
bounded above by the complexity of the function we denote by SBC (for sum of binomial coefficients), 
closely related to the binomial distributions with parameter $p = 1/2$. 
We also initiate the study of some additional fine properties of admissible permutations.
\end{abstract}

\maketitle

\makeatletter

\section{Introduction}
\label{sec:1}
\subsection{Motivation}
\label{subsec:1.1}
Let $\left\{ R_i\vert i \in \mathbb{Z}^+\right\}$ be an i.i.d sequence of random variables, each with mean 0 and variance 1, and for each $n > 0$, let $\displaystyle{S_n = \sum_{i=1}^n R_i}$ be the $n^{\text{th}}$ partial sum; thus, by the classical
Central Limit Theorem, the sequence
$\left\{ S_n/\sqrt{n}\right\}$ converges weakly to the standard normal.  It is well known, {[}\ref{Skorokhod}{]}, that that there is another sequence, $\left\{S_n^*\right\}$, defined on $(0,1)$ (equipped with Lebesgue measure), with each $S^*_n$
equivalent in distribution to $S_n$, such that $\left\{ S^*_n/\sqrt{n}\right\}$ converges almost surely to $Z$, the standard normal.  However, is there a general method for obtaining, for each $n$, an i.i.d. family $\left ( R^*_{n,i}\vert 1 \leq i \leq n\right )$, with each $R^*_{n,i}$ equal in distribution to $R_i$, and whose sum is equal (literally, pointwise, not just in distribution) to $S^*_n$?  Once such families are obtained, they provide
a {\it strong }triangular array representation of $\left\{S^*_n\right\}$.  This is the natural counterpart,
for almost sure convergence, of the standard notion of triangular array representation.

At this level of generality, the problem appears to be quite difficult.  This paper begins the investigation of the problem in the simple but important setting where the sequence $\left\{S_n\right\}$ is the (equiprobable) random walk on $\mathbb{Z}$ with domain $(0,1)$.  For $0 < x < 1$,  write
$x = \sum_{i=1}^\infty \varepsilon_i(x)2^{-i}$, with each $\varepsilon_i(x) \in \{ 0,1\}$ and
$\varepsilon_i(x) = 0$ for infinitely many $i$.  For $1 \leq i \leq n$ and $0 < x < 1$, we take 
$\displaystyle{R_i(x) = (-1)^{1+\varepsilon_i(x)}}$; this gives the simplest expression for the $S_n$. Even in this simple setting, 
the behavior of the $S_n/\sqrt{n}$ is extremely chaotic and the disorder increases with $n$.  One manifestion of this chaos is an easy consequence of the
LIL, {[}\ref{Khintchine}{]}, {[}\ref{Kolmogoroff}{]}  for example: for almost all $x,\ \overline{\lim}_{n \to \infty}S_n(x)/\sqrt{n} = \infty$ and $\underline{\lim}_{n \to \infty}S_n(x)/\sqrt{n} = -\infty$.  On the other hand, each $S^*_n$ is a non-decreasing step
function, mirroring the almost sure convergence of the sequence of
normalizations.  

The graphs of $S_n$ and $S^*_n$ for $n=6,7$ are
included  below. $S_n$ is shown in magenta, while
$S^*_n$ is shown in green.  

\begin{figure}[ht]
\begin{minipage}[b]{0.45\linewidth}
\centering
\includegraphics[width=\textwidth]{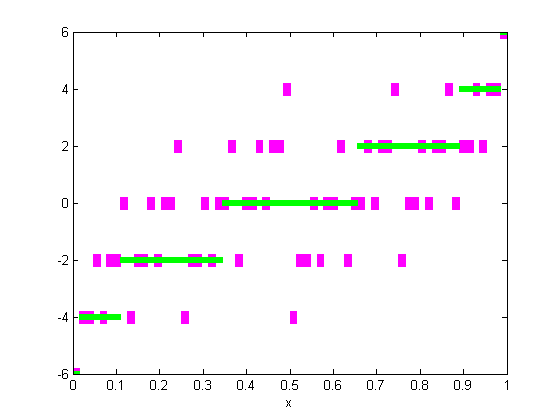}
\caption{$S_6,S^*_6$ }
\label{fig:figure1}
\end{minipage}
\hspace{0.5cm}
\begin{minipage}[b]{0.45\linewidth}
\centering
\includegraphics[width=\textwidth]{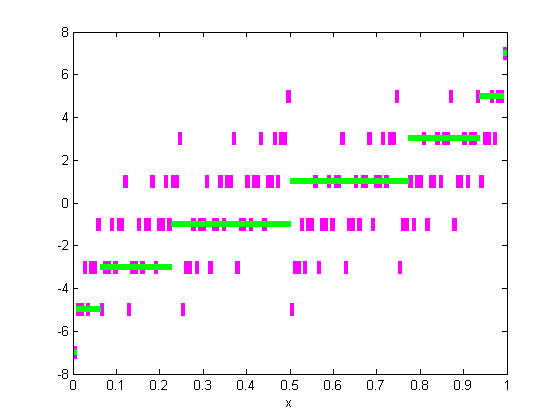}
\caption{$S_7,S^*_7$}
\label{fig:figure2}
\end{minipage}
\end{figure}

Since $S^*_n$ is equal in distribution to $S_n$, it follows that  
$S^*_n$ does have triangular array representations. But what
about {\it strong }ones, in the sense of the first paragraph?
Specializing to this setting, the problem laid out in that paragraph
can be restated as follows.
\begin{question}  How do we obtain {\it strong }triangular array representations,
$\left (R^*_{n,i}\vert n \in \mathbb{Z}^+,\ 1 \leq i \leq n\right )$,
of $\left\{S^*_n\right\}$?  What do the $R^*_{n,i}$ look like, as functions of $x$? 
\end{question}
The LIL has implications, here, as well.  For example, for each $n$,
at least one of the $R^*_{n,i}$ must depend on more than one bit.  Also, 
the $R^*_{n,i}$ must depend on $n$, not just on $i$.

We will give an explicit procedure for obtaining the sought-after $R^*_{n,i}$,  and they will have
an additional property; they are {\it trim } in that, for fixed $n$, each $R^*_{n,i}(x)$ will depend only on $\left (\varepsilon_1(x), \ldots , \varepsilon_n(x)\right )$.  This leads naturally to the next question.
\begin{question}
What are the {\it trim, strong }triangular array
representations of  $\left \{ S^*_n\right \}$?
\end{question}
Our procedure starts from any sequence $\left\{ \pi_n\right\}$,
where each $\pi_n$ is a permutation of $\left\{ 0, \ldots , 2^n-1\right\}$ with the additional property of being {\it admissible}.  Any permutation, $\pi$, of  $\left\{ 0, \ldots , 2^n-1\right\}$, can be viewed as permuting the level $n$ dyadic intervals (by permuting their indices).  This provides a rearrangement of $(0,1)$.  Such a permutation is {\it admissible } iff $S^*_n$ results when
the corresponding rearrangement is followed by applying $S_n$ (as a function).  This is made precise in Equation (1) of subsection
(1.2).

For each $n$, the passage from $\pi_n$ to $\left (R^*_{n,i}\vert 1 \leq i \leq n\right )$ is explicit, canonical and one-to-one and is given by Equation (2) of the proof of Theorem 1 in (2.2).  Further,
as $\pi_n$ varies over admissible permutations, our procedure generates {\it all }possible suitable  families $\left (R^*_{n,i}\vert  1 \leq i \leq n\right )$, where
each $R^*_{n,i}$ is trim.  The passage from $\left (R^*_{n,i}\vert 1 \leq i \leq n\right )$
to $\pi_n$ is given by Equation (3), also in (2.2).   The obvious extension to a canonical bjection between
sequences of admissible permutations and trim strong triangular array representations of $\left\{S^*_n\right\}$ is
given by Corollary 1.  Thus Theorem 1 and Corollary 1 answer Questions 1 and 2.  
 
Since almost sure convergence is such a restrictive condition, it is natural to ask how hard it is to produce the trim
strong triangular array representations of $\left\{S^*_n\right\}$ and what additional special properties they must have.
As with the existence of trim  strong triangular array representations, prior to this paper, very little was known; 
to our knowledge, the questions in the previous sentence have not been considered until now.  Once we know how to associate sequences of admissible permutations to trim strong triangular array representations, it becomes natural to pursue these questions in terms of the complexity of the associated sequences.  The second half of the paper carries out  such a complexity analysis, motivated by the following
questions.
\begin{question}
Are there trim, strong triangular
representations of $\left \{ S^*_n\right \}$
of low complexity? 
\end{question}
\begin{question}
Are there trim, strong triangular
array representations of $\left \{ S^*_n\right \}$
which differ as little as possible from the above representation
of $\left \{ S_n\right \}$ and which are also of low complexity?    
\end{question}
We explicitly construct three trim, strong triangular array representations of quite low complexity, as measured by the complexity
of their classifying sequences of admissible permutations, $\left\{ F_n\right\},\ \left\{ G_n\right\}$ and
$\left\{ H_n\right\}$.  Our basic complexity estimate is that they are all polynomial
time computable (we make this precise in subsection (1.2)).  This is the content of Theorem 2, in subsection (3.4), for 
$\left\{ F_n\right\}$ and of Theorem 3, in subsection (4.2), for $\left\{ G_n\right\}$ and $\left\{ H_n\right\}$.  These results therefore answer Question 3 affirmatively.  Since $\left\{ G_n\right\}$ and $\left\{ H_n\right\}$ are constructed so as
to differ as little as possible from the representation of $\left\{S_n\right\}$ by $\left\{R_i\right\}$, Theorem 3 answers Question
4 affirmatively.
 
The proofs of Theorems 2 and 3 yield the somewhat sharper result that each of  $\left\{ F_n\right\},\ \left\{ G_n\right\}$ and
$\left\{ H_n\right\}$ is very simply computed in terms of the function we denote by SBC (for Sum of Binomial Coefficients) introduced in Definition 1 in subsection (3.1).  This function is very naturally associated with the binomial distributions with parameter $p = 1/2$.    The computation of each of the three sequences in terms of SBC is a counterpart to the result that $2^n$ has a very simple expression in terms of {\it any} sequence, $\left\{\pi_n\right\}$, of admissible permutations.  Thus, the complexity of each of  $\left\{ F_n\right\},\ \left\{ G_n\right\}$ and $\left\{ H_n\right\}$ is bracketed in the fairly narrow range between that of $2^n$ and that of SBC.  We shed further light on the relationship between SBC and $\left\{ F_n\right\}$ in Corollary 4 of subsection (3.5); discussion is deferred until (1.3.5) and Section 3. 

While their properties are indeed rather special, the trim strong triangular array representations 
corresponding to  $\left\{ F_n\right\},\ \left\{ G_n\right\}$ and
$\left\{ H_n\right\}$ are, perhaps surprisingly, not so diffiicult to produce, since they are of low complexity.
At least in this context, the ``cost'' of the passage from $\left\{R_i/\sqrt{n}\right\}$ and the weakly covergent 
$\left\{ S_n/\sqrt{n}\right\}$ to the (trim) strong triangular array representations of the almost surely convergent $\left\{S^*_n/\sqrt{n}\right\}$ turns out to be surprisingly modest.  Progress has been made in the direction of extending our methods to a more general setting.  This will be the subject of a planned sequel to this paper.

This paper grows out of Chapters 3 and 4 
of Skyers' dissertation, {[}\ref{Skyers}{]}, 
written with Stanley as advisor.   \text{Dobri\ensuremath{\acute{\text{c}}}}, 
served as a {} ``co-advisor'', and provided the 
inspiration and impetus for the entire project.  
For recent work related to Skorokhod's work, 
in a rather different vein than this paper, see {[}\ref{Berti1}{]}, {[}\ref{Berti2}{]}.
\subsection{Preliminaries, Notation, Conventions}
\label{subsec:1.2}
Let $X$ be a random variable (on {\it any} probability space).  Let $F_X$
be the cumulative distribution of $X$.  By the {\it quantile of } $X$ (denoted
 by $X^*$), we mean the random variable on $(0,1)$
(equipped with Lebesgue measure, $\lambda$, on Borel sets, $\mathcal{B}$) defined
by:
\[ 
X^*(x) := \text{inf} \left \{ t \in \mathbb{R}\vert F_X(t) \geq x\right \}.
\]
\noindent
It is well-known that $X$ and $X^*$ are equal in distribution.
Skorokhod, {[}\ref{Skorokhod}{]}, showed that if 
$\left\{X_n\right\}$ is a sequence of random variables (on {\it any} probability
space) converging weakly to $X$, then the sequence of quantiles,
$X^*_n$, converges almost surely to $X^*$.  

In order to compare the structure of the initial sequence to that 
of the sequence of quantiles, the probability space of the $X_n$ should be 
$((0,1),\ \mathcal{B},\ \lambda )$, as above.  
In what follows, we work exclusively
in this probability space.  

In this paper, $i, j, k, m, n$ will \emph{always} denote non-negative
integers (elements of $\mathbb{N}$).  Most often, we will
have $n > 0$. 
We use $\vert X \vert$ for the 
cardinality of a set, $X$. 
To emphasize that
a union is a {\it disjoint union } we use $\sqcup$ or $\bigsqcup$
rather than the usual $\cup$ or $\bigcup$.  When the nature of
the index set is clear or has been established, we use
$\{a_i\}$ to denote the sequence (possibly finite, possibly multi-indexed)
whose term, for index $i$, is $a_i$.  We use $\chi_{Y,U}$ to denote the characteristic
(or indicator) function of a set $Y$, viewed as a subset of an ambient set $U$, the domain of the
characteristic function.  When $U$ is clear from context, we will omit it in the subscript.  This
notation is intended to cover the situation where $Y$ is a relation, i.e. where $U$ is a set
of $d-\text{tuples}$ for some fixed $d > 1$.

For integers, $n > 0$, and $0 \leq k < 2^n,\ D_{n,k}$ denotes the
$k^{\text{th}}$ level $n$ dyadic interval:  
\[
D_{n,0} = \left ( 0,2^{-n}\right ),
\text{and for } 0 < k < 2^n,\ D_{n,k} = \left [ 2^{-n}k, 2^{-n}(k+1)\right ).
\]
Note that both $S_n$ and $S^*_n$ are constant on each level $n$
dyadic interval.
A permutation, $\pi$,
of $\left \{ 0, 1, \ldots, 2^n-1\right \}$ is {\it admissible} if 
\begin{equation}
\text{for all } k \in  
\left \{ 0, 1, \ldots, 2^n-1\right \}, \text{ all } x \in D_{n,k} 
\text{ and all } y \in D_{n,\pi (k)},\ S^*_n(x) = S_n(y).
\end{equation}

In several places, we will have a function, $\phi$, with domain $(0,1)$, which, for some $n$,
is constant on each of the $D_{n,k}$.  We then use $\text{I}\phi$ (``$\text{I}\phi$'' for
the integer version of $\phi$) to denote the function with domain $\left\{ 0, \ldots , 2^n - 1\right\}$,
whose value at $k$ is the constant value of $\phi$ on $D_{n,k}$.   Thus, for example, $\text{I}S_n$ 
and $\text{I}S^*_n$ will denote the integer versions of $S_n$ and $S^*_n$, respectively.

We adopt a similar convention
for subsets, $X \subseteq (0, 1)$, such that $X$ is a union of level $n$ dyadic intervals.  We will
 then use $\text{I}X$ to denote the set of $k$ such that $D_{n,k} \subseteq X$.  
Strictly speaking, in both cases (function or subset) the dependence
on the specific $n$ involved should be part of the notation, but, in all instances,
this will already be incorporated into the notation used for the specific $\phi$ or $X$ involved.

Our basic complexity estimate is in terms of polynomial time computability.   This notion is robust across different
detailed models of computations, each of which has its own sensible notion of ``elementary operation''.  Accordingly,
as is customary, we omit a detailed development of what is involved in this notion.

If $f$ is a function of $d$ natural number arguments, $f$ is polynomial time computable (P-TIME)
iff for some polynomial, $p(n)$, the value of $f$ can be computed in at most
$p(n)$ elementary operations whenever all arguments are smaller than $2^n$. 
 This is consistent with the usual treatments ( e.g. {[}\ref{Clote}{]} or 
 {[}\ref{Papa}{]}), which, for the most part, treat arguments and values
as bitstrings or vectors of bitstrings, rather than in terms of the encoded natural
numbers.  In some important instances, $f$ will have $1+d$ arguments, the first of which is viewed as
being $n$ itself (the argument of $p$.)  In this context, the requirement is that at most $p(n)$ elementary
operations are required to compute $f\left ( n,x_1,\ldots , x_d\right )$ whenever each $x_i < 2^n$.
Polynomial-time decidable (P-TIME decidable) relations are ones whose characteristic functions are P-TIME.   
\subsection{Summary and Further Discussion of Results}
\label{subsec:1.3}
By Corollary 1, the existence of trim strong triangular array representations of $\left\{S^*_n\right\}$
reduces to the existence of sequences of admissible permutations, which further reduces to
the existence of admissible permutations of $\left\{ 0, \ldots , 2^n-1\right\}$, 
for each $n$.  This is established in part 2 of Lemma 1 of (3.1); the precise statement is that there are 
$\prod_{i=0}^n\left (\binom{n}{i}!\right )$ of them.  This count builds on a more
concrete characterization of admissibility developed, among other things, in (3.1).

In (3.2), Corollary 2 pulls together the statements of Theorem 1, Corollary 1 and part 2 of Lemma 1 
to give the existence proof for trim strong triangular array representations
of $\left\{S^*_n\right\}$.  Proposition 1 builds on Corollary 2 by constructing a strong but non-trim triangular array
representation starting from a trim strong one.  Proposition 2, which also builds on some of the material from (3.1), is
a ``non-persistence'' result in that it shows that in any sequence $\left\{\pi_n\right\}$ of admissible permutations,
$\pi_n$ never persists to be a subfunction of $\pi_{n+1}$, and that in any trim strong triangular array representation
$\left ( R^*_{n,i}\vert n > 0,\ 1 \leq i \leq n\right )$ of $\left\{S^*_n\right\}$, for any $n > 0$, it is never possible for all of the
$R^*_{n,i}$ to  persist to be the corresponding $R^*_{n+1,i}$.  This provides another proof of the second consequence of
the LIL mentioned following Question 1 in (1.1) and highlights some important ways in which the trim strong triangular array representations of $\left\{S^*_n\right\}$ must differ from the simple representation $\left\{ R_i\right\}$ of $\left\{ S_n\right\}$.
\subsubsection{The role of trimness}
\label{subsubsec:1.3.1)}
Proposition 1 shows that trimness does not ``come for free''.  Given that there can be no strong triangular array representation for $\left\{ S^*_n\right\}$ in which each $R^*_{n,i}$ depends on only one bit, trimness is a natural ``next best hope''.
Its central role in Theorem 1 and Corollary 1 is further evidence for its naturality, as is the following equivalent characterization of trimness, suggested by A. Nerode.

 Let $V_n$ be 
$\{ -1 , 1\}^n$.
Let $\mathcal{V}$ be the topological product of the $V_n$ equipped with the discrete topologies and let $V$ be the set of
points of $\mathcal{V}$.  Suppose
that for $n \geq i,\ R'_{n,i}$ is a Rademacher random variable on $(0,1)$ (we are not
necessarily assuming, yet, that $\left\{ R'_{n,i}\right\}$ is a triangular array nor that it
represents $\left\{ S^*_n\right\}$).  This provides us with a transformation, $T$, from $(0,1)$ to
$V$, by taking $T(x)_n := \left ( R'_{n,1}(x), \ldots , R'_{n,n}(x)\right )$.  It is then clear that
the condition:
\[
\text{for all }i \leq n,\ R'_{n,i}\text{ depends  at most on }\varepsilon_1, \ldots \varepsilon_n
\]
is equivalent to the condition:
\[
\text{the associated transformation }T\text{ is Lipschitz-continuous with } \delta = \epsilon.
\]
Then, specializing to the situation where the $R'_{n,i}$ {\it do } furnish a strong triangular array
representation of $\left\{ S^*_n\right\}$, the second displayed formula provides our equivalent
characterization of trimness.
We are grateful to A. Nerode for suggesting that we seek this type of characterization, and for
the observation that the transformation $T$ can be feasibly implemented, since the implementation
would satisfy a strong form of bounded memory; this is an equivalent refomulation of the Lipschitz
continuity.   
\subsubsection{Transition to Complexity:  $2^n$ is a lower bound.}
\label{subsubsec:1.3.2)}  
In (3.3) we introduce the natural encoding, $\Pi$ of a sequence, $\left\{\Pi_n\right\}$, of admissible permutations,
and explicitly begin to deal with complexity issues.  It is here that we really begin
to exploit the ``toolkit'' material developed in (3.1).
We prove Proposition 3,
which establishes that the exponential function, $2^n$, is a lower bound for the complexity
of{\it\ any } such sequence $\left\{\Pi_n\right\}$.  Our approach to this,
and to related questions, will be briefly outlined  in (1.3.5), below and more fully discussed at the start of Section 3.
\subsubsection{Preview of Theorem 2}
\label{subsubsec:1.3.3}
While Corollary 2 settles the question of the existence of trim strong
triangular array representations of $\left \{ S^*_n\right \}$, in Definition 4 ((3.4)) we
explicitly construct the sequence $\left \{ F_n\right \}$ of admissible permutations.
Building on Proposition 4 and Corollary 3, Theorem 2 answers Question 3 affirmatively by establishing that $F$, the natural encoding of this sequence, is P-TIME and can be simply computed in terms of SBC.  We also give an even cleaner and simpler defining expression for $2^n$ in terms of $F$, improving slightly on the proof of Proposition 3.  The discussion of Corollary 4,
proved in (3.5), is deferred until (1.3.5) and Section 3.
\subsubsection{Preview of Theorem 3}
\label{subsubsec:1.3.4}
The affirmative answer to Question 4 is provided by Theorem 3, proved in (4.2).  In (4.1), culminating in Definition 7, 
we construct $\left \{ G_n\right \}$, a variant  of the sequence, $\left \{ F_n\right \}$.  Among
admissible permutations of $\left\{ 0, \ldots, 2^n-1\right\}$, 
$G_n$ is maximal for agreement with the identity function. 
Thus, the extent to which trim strong triangular array representations of
$\left\{ S^*_n\right\}$  {\it must} differ from the canonical representation
of $\left\{ S_n\right\}$ is measured by the extent to which the $G_n$ differ from the identity permutations.

The construction of $\left\{ G_n\right\}$ is also motivated
by a rather different notion of complexity: that of an individual admissible permutation.  This also
motivates the construction of $\left\{ H_n\right\}$,
the other variant of $\left\{ F_n\right\}$ 
(Definition 11 of (4.1)).  We impose additional 
natural properties on the $G_n$ and $H_n$ 
to guarantee that their orbit structures will be 
simpler than the orbit structures of the $F_n$.
While Remark 6 and Definition 7 immediately
make it clear that each $G_n$ is admissible, this is a more substantial issue for the $H_n$ and is
established in Proposition 5 of (4.1), the analogue for $\left \{ H_n\right \}$ of Remark 6.

The analogue of Theorem 2 for the
natural encodings, $G$, of $\left\{ G_n\right\}$, and
$H$, of $\left\{ H_n\right\}$ is provided by Theorem 3.  The expression for $2^n$ in
terms of $G$ is fairly close to the one in terms of $F$, 
but for $H$, we content ourselves with the general lower bound
statement of Proposition 3.  This difference between $F$ and $G$,
on the one hand, and $H$, on the other, is foreshadowed by the discussion
at the end of (3.5), and revisited at the end of (4.1).  Similar issues, also discussed
at the end of (4.1), are obstacles to obtaining a genuine analogue of Corollary 4
for $G$ or $H$.
The proof that each of $G$ and $H$ is P-TIME and is simply and explicitly computed in
terms of SBC proceeds by analogy to Theorem 2,
and follows the general approach sketched in (1.3.5).  
This builds on Proposition 6, which plays the role of 
the combination of Proposition 4 and Corollary 3. 
\subsubsection{Complexity Issues} 
\label{subsubsec:1.3.5}
For the lower bounds, established by Equation (6) of
Proposition 3, Equation (8) of Theorem 2, and
Equation (13) of Theorem 3, the approach is to express $2^n$ 
explicitly and uniformly in $n$, using the sum of at most $n$ values
(including repeated values) of the function involved.  All of the 
corresponding arguments are 
obtained from $n$, very simply, explicitly, and uniformly in $n$.
In Equation (6), there are $n$ distinct values involved, 
and the sum of these values is incremented by 1 to obtain $2^n$.  In Equation (8), for $F$,
there is just one value, repeated twice.  Finally, for Equation (13), for $G$, 
there is also just one value, repeated twice, but that value is the maximum of
two values.  In both of the latter cases, $2^n$ is simply the sum of the repeated values (i.e.,
twice the repeated value, but the point is to do things using only addition and no multiplication).

The proofs that $F,\ G$ and $H$ are P-TIME 
have a common structure, though the cases
in the definitions of $G,\ H$ result in technical complications,
especially for $H$.  Here, accordingly, it is only for $F$ (Theorem 2)
that we will outline the main ideas of the proof.

In subsections (3.1) and (3.4), we introduce the function SBC (in Definition 1) and 
a number of other auxiliary functions and relations, notably the two  functions, $\text{IStep}$ 
and $\text{EW}$.  The function $\text{IStep}$
computes ``positions'' with respect to $\text{SBC}$ and is introduced in
the comments following Remark 1, in (3.1).  We introduce
$\text{EW}$ in Definition 5 of (3.4).  It is the enumerating function for
$\text{Weight}$ (Definition 2:  the usual Hamming weight of a positive integer).  The 
sequence http://oeis.org, 2010, Sequence A066884,  {[}\ref{OEIS}{]}, 
encodes $\text{EW}$.  $\text{IStep}$
bears the same relationship to $\left\{ S^*_n\right\}$ 
as $\text{Weight}$ does to $\left\{ S_n\right\}$.

In subsection (3.4), 
Proposition 4 establishes, among other things, that 
for all $n$ and all $i \leq n$ the computation of 
$\text{SBC}(n,i)$ can be carried out using
$O\left ( n^2\right )$ additions of  integers all
below $2^n$, with all of the intermediate sums being less than $2^n$.
This guarantees that the function $\text{IStep}$ is
P-TIME.  

The corresponding result for the function $\text{EW}$ is
given by Corollary 3 which also establishes that the relation expressed 
by Equation (7 ) is P-TIME decidable.  Corollary 3 is the culmination of a 
sequence of results (Proposition 4, Lemmas 2 and 3)
where the notion of  ``tame'' relation, introduced in Definition 3, is a key ingredient. 
Lemma 2 establishes that the enumerating function (viz. Definition 3) for a tame relation is P-TIME.
Lemma 3 establishes the tameness of the relation whose enumerating function is $\text{EW}$ 
(and thus, with Lemma 2, that $\text{EW}$ is P-TIME).  It follows from the proof of Lemma 3
that $\text{EW}$ itself can be simply computed in terms of SBC and that solutions, $m$, of
Equation 7 can be computed in polynomial time as functions of $n$ and $k$.

The proof of Theorem 2 then proceeds by showing that  for $n > 0$ and 
$k < 2^n,\ F(n,k) = m < 2^n$ is the unique 
solution of Equation (7).    We are very grateful to S. Buss 
who provided invaluable assistance on several occasions.  In particular,
he confirmed the validity of our approach to Proposition 4, pointed out 
the references given there and supplied a very nice
argument that evolved into the proof of Lemma 3.  He also suggested combining
these ideas with binary search (a variant of the Bisection Algorithm).  In the course of working
out the details, we isolated the notion of tameness and formulated and proved Lemma 2.
A bit more detail on the material of this subsubsection is given in the overview of Section 3.
\section{Theorem 1 and Corollary 1}
\label{sec:2}
\subsection{Preliminaries for Theorem 1}
\label{subsec:2.1}
For $x \in (0,1)$, we set: 
\[
\mathbf{r}_n(x) := 
\left ( \varepsilon_1(x), \ldots , \varepsilon_n(x)\right ),
\]
\[
\mathbf{s}_n(x) := 
\left ((-1)^{1+ \varepsilon_1(x)}, \ldots , (-1)^{1+\varepsilon_n(x)}\right ).
\]
On $D_{n,k},\ \mathbf{r}_n(x), \mathbf{s}_n(x)$  are constant.  Therefore, following
the convention in the final paragraph of (1.3), we denote 
these constant values by $\text{I}\mathbf{r}_n(k),  \text{I}\mathbf{s}_n(k)$,
respectively.

Note that $\text{I}\mathbf{r}_n(k)$ does NOT denote the kth component of a vector,
$\text{I}\mathbf{r}_n$, rather it denotes the vector (a length
$n$ bitstring) itself.  We will denote the $i^{\text{th}}$ component
of this vector by $\left ( \text{I}\mathbf{r}_n(k)\right )_i$.  Note that
this is just $\varepsilon_i(x)$ for any $x \in D_{n,k}$.  Similar observations
hold with $\mathbf{s}$ in place of $\mathbf{r}$ (and $\{ -1, 1\}$ replacing $\{ 0, 1\}$).

Note, further, that
$\text{I}\mathbf{r}_n(k)$ is the reversal of the binary representation of 
$k$:\ \ $\displaystyle{k = \sum_{i=1}^n 2^{n+1 - i}\left ( \text{I}\mathbf{r}_n(k)\right )_i}$, 
and that $\left ( \text{I}\mathbf{r}_n(k)\vert\, k < 2^n\right )$ 
enumerates $\{ 0, 1\}^n$ in increasing order with respect to the lexicographic 
ordering.  For $\mathbf{r} \in \{ 0, 1\}^n$, we also let:
\[ 
D_{\mathbf{r}} := \{ x \in (0,1)|\mathbf{r}_n(x) = \mathbf{r}\}.  
\]
Letting $k$
be such that $\mathbf{r} = \text{I}\mathbf{r}_n(k)$, we note that $D_{\mathbf{r}} = D_{n,k} = 
\left\{ x \in (0,1)\vert \mathbf{r}_n(x) = \text{I}\mathbf{r}_n(k)\right\}$.
 
We use $\nu_n$ to denote the order isomorphism
(with respect to lexicographic order) between 
$\{ 0, 1\}^n$ and
$\{ -1, 1\}^n$, thus for 
$\mathbf{r} \in
\{ 0, 1\}^n$ and $1 \leq i \leq n,\ 
\left ( \nu_n(\mathbf{r})\right )_i = (-1)^{1+(\mathbf{r})_i}$. 
Note that $\text{I}\mathbf{s}_n(k) = \nu_n\left ( \text{I}\mathbf{r}_n(k)\right )$.
It is also worth noting that \emph{any} permutation, $\pi$, of $\left\{ 0, \ldots, 2^n-1\right\}$
is naturally viewed as a permutation of $\{ 0,1\}^n$, by taking
$\pi\left ( \text{I}\mathbf{r}_n(k)\right )$ as defined to be $\text{I}\mathbf{r}_n(\pi(k))$, and similarly
with $\{ -1, 1\}$ replacing $\{ 0, 1\}$ and $\mathbf{s}$ replacing
$\mathbf{r}$.
\subsection{Theorem 1 and Corollary 1}
\label{subsec:2.2}
\begin{theorem}  For each $n$, there is a canonical bijection
between admissible permutations of 
$\left\{ 0, \ldots , 2^n-1\right \}$ and representations of $S^*_n$ as a sum
\[
S^*_n = \sum_{i=1}^nR^*_{n,i},
\]
where $\left ( R^*_{n,i}\vert\, 1 \leq i \leq n\right )$ is an
i.i.d. family of Rademacher random variables each of which 
has mean 0, variance 1 and depends 
only on $\mathbf{r}_n$.  
\end{theorem}
\begin{proof}
Fix $n > 0$.  We first construct the $R^*_{n,i}$, given $\pi$, and then construct $\pi$ given the  
$R^*_{n,i}$.  We then carry out the necessary verifications in each direction. 

First, let $\pi$ be an admissible permutation of
$\left\{ 0, \ldots , 2^n-1\right\}$.  For
$x \in (0, 1)$, let $k$ be such that $x \in D_{n,k}$, and let $1 \leq i \leq n$.    
Then:    
\begin{equation}
\text{let } y \text{ be any member of } D_{n,\pi (k)} \text{ and define:  }  R^*_{n,i}(x) := (-1)^{1+ \varepsilon_i(y)}.
\end{equation}
\noindent

Conversely, given an independent family, $\left ( R^*_{n,i}\vert 1 \leq i \leq n\right )$, such that
$\displaystyle{S^*_n = \sum_{i=1}^nR^*_{n,i}}$, where
each $R^*_{n,i}$ is Rademacher with mean 0 and variance 1 and depends
only on $\mathbf{r}_n(x)$, we obtain $\pi$ as follows.  Given $k < 2^n$, let
$x \in D_{n,k}$, let $\mathbf{s} = \left ( R^*_{n,i}(x)\vert 1 \leq i \leq n\right )$ and
define:
\begin{equation}
\pi (k)\ =\ \text{that } m < 2^n\ \text{such that } \mathbf{s} = \text{I}\mathbf{s}_n(m).
\end{equation}
Clearly these constructions yield a bijection, so we turn to the necessary verifications.

First suppose $\pi$ is admissible and that the 
$R^*_{n,i}$ are defined by Equation (2). 
Clearly these $R^*_{n,i}$ depend only on $\mathbf{r}_n(x)$.
In order to see that they sum to $S^*_n$, note that:
\[
\text{for all } k < 2^n \text{ and all } x \in D_{n,k},\
S^*_n(x) = S_n(y) \text{,  for any } y \in D_{n,\pi (k)},
\]
\[
\text{i.e. } S^*_n(x)
= \sum_{i=1}^n(-1)^{1+\varepsilon_i(y)} \text{, for any such } y,
\text{i.e. } S^*_n(x) = \sum_{i=1}^nR^*_{n,i}(x);\text{ this suffices.}
\]

To see that each $R^*_{n,i}$ is Rademacher with mean 0 and variance 1, 
fix $i$ and  $\epsilon\in\left\{ 0,1\right\}$  and let 
$A:=\left\{ \mathbf{t}\in\left\{ 0,1\right\} ^n\big\vert 
t_i=\epsilon\right\}$; then
$\vert A\vert = 2^{n-1}$. Since $\pi$
is 1-1, $\left\vert \pi^{-1}\left[A\right]\right\vert = 2^{n-1}$. Now, 
viewing $\pi$ as a permutation of $\{0, 1\}^n$, we have that:
\[
\pi^{-1}\left[A\right] = 
\left\{ \mathbf{r}\in\left\{ 0,1\right\} ^n\big\vert
\left( \pi\left( \mathbf{r}\right)\right)_i=
\epsilon\right\} \text{ and }
\left\{ x\vert\left( \pi\left( \mathbf{r}_n(x) \right)\right)_i=\epsilon\right\} =  
\bigsqcup_{\mathbf{r}\in \pi^{-1}\left[A\right]}
 D_{\mathbf{r}}.
\]
\noindent
It follows that:  
\[
\lambda\left(\left\{ x\vert\left( \pi\left( \mathbf{r}_n(x)\right )\right )_i =
\epsilon\right\} \right) =
\lambda\left(
\bigsqcup_{\mathbf{r}\in \pi^{-1}\left[A\right]}
D_{\mathbf{r}}\right ) = 2^{n-1}\cdot2^{-n}=1/2;\text{ this suffices.}
\]

In order to see that these $R^*_{n,i}$ are independent,
it suffices to show that:
\[
\text{for all } \mathbf{s} = \left ( s_1, \ldots , s_n\right ) \in\left\{ -1,1\right\}^n,\
p\left(s_1,\ldots,s_n\right) = 
p_1\left(s_1\right)\cdot\ldots\cdot p_n\left(s_n\right),
\]
where $p$ is the joint pmf of the $R^*_{n,i}$ and $p_i$
is the pmf of $R^*_{n,i}$ alone.  
We showed that $p_1\left(s_1\right)\cdot\ldots\cdot p_n\left(s_n\right) =
2^{-n}$, so,
again viewing $\pi$ as a permutation
of $\{ 0, 1\}^n$, let 
$\mathbf{r} := \pi^{-1}\circ \left ( \nu_n\right )^{-1}(\mathbf{s})$ and note that:
\[
P\left(R^*_{n,1}=s_1,\ldots,R^*_{n,n}=s_n\right) = 
\lambda\left(\left\{ x\vert \pi\left ( \mathbf{r}_n(x)\right ) = 
\nu_n^{-1}(\mathbf{s}\right\} \right) = 
\lambda\left(D_{\mathbf{r}}\right) = 2^{-n}.  
\]

For the opposite direction, suppose that $\left ( R^*_{n,i}\vert 1 \leq i \leq n\right )$
is given with the stated properties.  Let $\pi$ be defined
by Equation (3). 
We first show that $\pi$ is one-to-one.  For this, let $x\in D_{n,k},\ 
\mathbf{s} = \left ( R^*_{n,i}(x)\vert 1 \leq i \leq n\right )$ and note that if
$\mathbf{u} \in \{ 0, 1\}^n$ is such that 
\[
\text{for } y \in D_{\mathbf{u}},\ \left ( R^*_{n,i}(y)\vert 1 \leq i \leq n\right ) =
\mathbf{s} \text{, then } \mathbf{u} = \text{I}\mathbf{r}_n(k).  
\]
If this
were to fail we would have that 
\[
P\left(R^*_{n,1}=s_1,\ldots,R^*_{n,n}=s_n\right) \geq\  
\lambda\left(D_{\mathbf{u}}\right)+
\lambda\left(D_{\text{I}\mathbf{r}_n(k)}\right)=2^{-n+1}, 
\]
which contradicts
our hypotheses on the $R^*_{n,i}$.  Thus, $f$ is one-to-one.  Admissibility
then follows, because now, by hypothesis, if $x \in D_{n,k}$ and $m = \pi (k)$, then:
\[
S^*_n(x) = \sum_{i=1}^nR^*_{n,i}(x) = \sum_{i=1}^ns_i \text{, but also 
for any } y \in D_{n,m},\ S_n(y) = \sum_{i=1}^ns_i,
\]
\noindent
as required.  
\end{proof}
\begin{corollary}  There is a canonical bijection between
sequences, $\left\{ \pi_n\right\}$,
of admissible permutations of $\left\{ 0, \ldots , 2^n-1\right \}$ and trim,
strong triangular arrays for 
$\left\{ S^*_n\right\}$.
\qed
\end{corollary}
Given $n$ and $\left ( R^*_{n,i}\vert 1 \leq i \leq n\right )$, Equation (3) 
is best seen as as a finer version of Equation (1) (the definition of admissible permutation).
Incorporating the additional information in the representations $\left ( R_i\vert 1 \leq i \leq n\right )$
and $\left ( R^*_{n,i}\vert 1 \leq i \leq n\right )$ singles out
a specific admissible permutation, whereas Equation (1) defines the set  of all of them.
Equation (2) reverses this, taking as given the canonical representation, 
$\left ( R_i\vert 1 \leq i \leq n\right )$, together with a specific admissible permuation 
and singling out a specific representation,
$\left ( R^*_{n,i}\vert 1 \leq i \leq n\right )$, of $S^*_n$.
\section{The Road to Theorem 2}
In (3.1) we introduce the notions that will provide
the ``toolkit'' for the rest of the paper, and, in particular
for Theorems 2 and 3.  In Definition 1, we introduce
the functions $\text{Step}_n,\ \text{Weight}_n$ and SBC.
Incorporating Remark 1 then immediately gives us the functions $\text{IStep}_n,\ \text{I}S\text{ and I}S^*$.  These are
the ``integer versions'' of the functions $\text{Step}_n,\ S_n\text{ and }S^*_n$, respectively.  We also
introduce the function $\text{IStep}$ which is the natural encoding
(as a function of two variables) of the family of the 
$\text{IStep}_n$.   

In Definition 2, we introduce the sets $A_{n,i}$ and $B_{n,i}$ and
their ``integer versions'' $\text{I}A_{n,i}$ and
$\text{I}B_{n,i}$.  Analogues of these are given
in (4.1) in the construction of the $\left\{ G_n\right\},\ 
\left\{ H_n\right\}$ and their natural encodings (see (3.3))
by the functions $G$ and $H$.

The $\text{I}A_{n,i}$ and $\text{I}B_{n,i}$
are the motivating paradigm for Definition 3, which is particularly important.  
It is given in abstract
form to accommodate five different invocations:  one
in (3.1) itself, in connection with the $\text{I}A_{n,i},\ 
\text{I}B_{n,i}$, and the other four in (4.1), in connection with
their analogues.  The invocation of Definition 3 in (3.1)
introduces the three-place relations $\text{RI}A$ and
$\text{RI}B$, which encode the families $\text{I}A_{n,i}$
and $\text{I}B_{n,i}$ respectively, their associated
cardinality functions, $\text{cd}_{\text{RI}A}$ and
$\text{cd}_{\text{RI}B}$ and enumerating functions,
$\text{ERI}A$ and $\text{ERI}B$.  The notion of
``tameness'' is also introduced in Definition 3; it  plays
an important role in the complexity analyses
carried out in subsections (3.4) and (4.2).  

The cardinality functions associated with P-TIME decidable
relations form the somewhat unusual complexity class, 
$\#P$ (see, e.g.,  {[}\ref{Clote}{]}, (6.3.5)), which, conjecturally,
is not included in the collection of  P-TIME functions.  In addition
to all of the other ways we appeal to tameness, it also guarantees (Item 6. of
Remark 3) that the relation itself is P-TIME decidable and that (by definition)
its associated cardinality function is P-TIME, allowing us to
avoid the issue of $\#P$.

We elaborate a bit on the sketch (in (1.3.3) and (1.3.5)) of what
is accomplished in subsections (3.4) and (3.5), since 
the material of  subsections (3.2) and (3.3) has
already been fully presented in subsection (1.3).
In (3.4), Proposition 4  establishes that 
for all $n$ and all $i \leq n$, the binomial coefficient, $\binom{n}{i}$
and $\text{SBC}(n,i)$ can be computed using $O(n^2)$ additions
of integers all below $2^n$ with all intermediate sums being less than
$2^n$ and so, in particular, in time polynomial in $n$.
Consequently, the function $\text{IStep}$ is
P-TIME, and the relation $\text{RI}A$ is tame.  This sets
the stage for Lemmas 2 and 3, Corollary 3 and Theorem 2.  In (3.5),
Corollary 4 sheds additional light on the relationship between $F$
and SBC by showing that SBC can be simply computed
in terms of $F$ and the function, which we denote by $\text{Inv}F$, which
is the natural encoding of the sequence $\left\{ F^{-1}_n\right\}$.  
We conclude (3.5) by arguing for the (informal and therefore
necessarily imprecise) thesis that $\left\{ F_n\right\}$ is,  in fact, is the {\it simplest}
sequence of admissible permutations.  The argument appeals to the view of admissible permutations
developed at the end of (3.1).
\subsection{Toolkit for Theorems 2 and 3}
\label{subsec:3.1}
As usual, $F_{B(n, p)}$ denotes the cumulative binomial distribution 
with parameters $n, p$.  
\begin{definition}   For $n >  0$, set $\text{SBC}(n,0) := 0$
and for $1 \leq i \leq n + 1$, set 
\[
\text{SBC}(n,i) := 
\sum_{j=0}^{i-1}\ \binom{n}{j}.
\]
\noindent
For $n > 0$, and $x \in (0,1)$,
$\displaystyle{\text{Step}_n(x)}$ is the unique $i$
such such that $F_{B(n, 1/2)}(i) \leq x < F_{B(n, 1/2)}(i+1)$.  We also set
$\displaystyle{\text{Weight}_n(x) := \sum_{i=1}^n\varepsilon_i(x)}$.
\end{definition}
\begin{remark}  For $x \in (0,1)$. 
and $n > 0$, the following observations are obvious:
\begin{enumerate}
\item  $S_n(x) = -n + 2\text{Weight}_n(x)$,

\item  $\text{Weight}_n(x)$ is the usual Hamming weight of $\mathbf{r}_n(x)$,

\item  $\text{Step}_n(x)$ is the unique $i$ such that $\text{SBC}(n,i) \leq x2^n < 
\text{SBC}(n,i+1)$ ,

\item  $S^*_n(x) = -n + 2\text{Step}_n(x)$.

\item  $\text{Step}_n(x),\ \text{Weight}_n(x)$ depend at most on $\mathbf{r}_n(x)$.\qed
\end{enumerate}
\end{remark}
In view of (1.3) and 5. of Remark 1, we have defined $\text{IStep}_n(k),\
\text{IWeight}_n(k)$.  We already knew that $S_n$ and $S^*_n$ also depend at most
on $\mathbf{r}_n(x)$ and thus we have also defined $\text{I}S^*_n(k),\ \text{I}S_n(k)$ for
$k < 2^n$; by the usual identification, we have also defined 
$\text{I}S^*_n(\mathbf{r}),\ \text{I}S_n(\mathbf{r})$ for
$\mathbf{r} \in \{ 0, 1\}^n$.  Also, $\text{IWeight}_n(k)$ is
the usual Hamming weight of the binary representation of $k$ and
is therefore independent of $n$, so henceforth this will simply
be denoted by $\text{Weight}(k)$.  Similarly 
$\text{Weight}(\mathbf{r})$ denotes the usual Hamming weight
of $\mathbf{r}$ for finite bitstrings, $\mathbf{r}$.
{\it In what follows, we shall use the notation }$\text{IStep}(n,k)$
rather than $\text{IStep}_n(k)$.  
The following is then also obvious 
\begin{remark}  For $n > 0$ and $k < 2^n$:
\begin{enumerate}
\item  $\text{IStep}(n,0) = 0$ and for $0 < k < 2^n ,\ 
\text{IStep}(n,k)$ is the least positive $i \leq n$ such that 
$\displaystyle{k < \text{SBC}(n,i+1)}$,

\item  For all $i \in \mathbb{N},\ 2^i-1$ is the least $k$ such that
$\text{Weight}(k) = i$ and $2^n - 2^{n-i}$ is the largest
$k < 2^n$ such that $\text{Weight}(k) = i$,

\item  For all $0 \leq i \leq n,\
\text{SBC}(n,i)$ is the least $k$ such that $\text{IStep}(n,k) = i.$\qed
\end{enumerate}
\end{remark} 
\begin{definition}  For   
$i \leq n$, we define $A_{n,i},\ B_{n,i}$ by:
\[
A_{n,i} :=
\{ x \in (0,1)\vert \text{Step}_n(x)\ = i\},
\]
\[ 
B_{n,i} := \{ x \in (0,1)\vert \text{Weight}_n(x)\ = i\}.
\]
In view of Remark 2, for fixed $n > 0$, each of the $A_{n,i},\ B_{n,i}$ is the 
union of level $n$ dyadic intervals, and therefore,
in view of the last paragraph of (1.3), $\text{I}A_{n,i},\ \text{I}B_{n,i}$ 
will be used to denote the corresponding subsets of $\left\{ 0, \ldots , 2^n-1\right\}$\ 
(or of $\{ 0, 1\}^n$, via the usual identification) :\
\begin{equation}
\text{I}A_{n,i} := \left\{ k < 2^n\vert D_{n,k} \subseteq A_{n,i}\right\}\text{ and }
\text{I}B_{n,i} := \left\{ k < 2^n\vert D_{n,k} \subseteq B_{n,i}\right\}.
\end{equation}
We also let $\alpha_{n,i} :=  \left\vert\text{I}A_{n,i}\right\vert,\
\beta_{n,i} :=  \left\vert\text{I}B_{n,i}\right\vert$ and for positive integers, $x < 2^n$,
we let $\alpha (n,i,x) =  \left\vert\text{I}A_{n,i} \cap \{1,\ldots , x\}\right\vert,\
\beta (n,i,x) =  \left\vert\text{I}B_{n,i} \cap \{1,\ldots , x\}  \right\vert$.
\end{definition}

This provides the motivating paradigm for the next definition, taking $d = 1,\
U_n = \{ 0, \ldots, n\}$ and $X_{n,i} = \text{I}A_{n,i}$ or $X_{n,i} = \text{I}B_{n,i}$.

\begin{definition} Suppose that $d \in \mathbb{Z}^+$ and that for $n > 0,\ U_n
\subseteq \{ 0, \ldots , n\}^d$.  
Suppose, further, that for $u \in U_n$, we have non-empty
$X_{n,u} \subseteq \left\{ 0, \ldots , 2^n -1\right\}$, with the
increasing enumeration of $X_{n,u}$ denoted by
$\left ( x_{n,u,s}\vert 1 \leq s \leq \left\vert X_{n,u}\right\vert \right )$.  
We define $\text{R}X(n, u, \ell )$ to be that $d+2-\text{place}$ relation
on $\mathbb{N}$ such that
\[
\text{R}X(n, u, \ell )\text{\ iff\ } \ell \in X_{n,u}.
\]
The {\it cardinality function associated with }$\text{R}X$ is the
function $\text{cd}_{\text{R}X}$ which, for $n >0,\ u \in U_n$
and positive integers, $x < 2^n$, assigns to $(n,u,x),\ \text{cd}_{\text{R}X}(n,u,x) :=
\left\vert X_{n,u} \cap \{1,\ldots , x\}\right\vert$.  
We denote by $\text{ER}X$ the enumerating function
for $\text{R}X$:  $\text{ER}X(n,u,s) := x_{n,u,s}$ for
$1 \leq s \leq \left\vert X_{n,u}\right\vert$. 
We will call $\text{R}X$ {\it tame } if $\text{cd}_{\text{R}X}$ is P-TIME.
\end{definition}

In our invocations of Definition 3,
we will have $d = 1$ and $U_n = \{ 0, \ldots , n\}$ (for the first
two invocations, and the last one) or $d = 2$
and $U_n =\{ (i,j) \vert i, j \leq n, i \neq j\}$ (for the third and fourth invocations).
In all of our invocations, it will be true that for fixed $n,\ \left\{ X_{n,u}\vert
u \in U_n\right\}$ will be a pairwise disjoint family, but we have not built this into
the definition of tame.

{\bf We now invoke Definition 3 with $d = 1$ and  $U_n = \{ 0, \ldots , n\}$
and with }$X_{n,i} = \text{I}A_{n,i}\ ${\bf or } $X_{n,i} = \text{I}B_{n,i}$.  
{\it This defines }$\text{RI}A,\ \text{RI}B,\ \text{cd}_{\text{RI}A},\
\text{cd}_{\text{RI}B},\ \text{ERI}A,\ \text{ERI}B$.
{\it The notation for the increasing enumerations will be }$a_{n,i,s},\ b_{n,i,s}$.
\begin{remark}  The following observations are immediate, with the exception of
item 6.
\begin{enumerate}
\item  For fixed $n, i, j$ we'll have that $\text{RI}A(n,i, \ell )$ holds 
$n > 0, 0 \leq i \leq n,\ 0 \leq  \ell < 2^n$ and $\text{IStep}(n,\ell ) = i$. 
Similarly, $\text{RI}B(n,j,\ell )$ holds iff $n > 0, 0 \leq j \leq n,\ 0 \leq \ell < 2^n$ and 
$\text{Weight}(\ell ) = j$.  
\item  $\text{cd}_{\text{RI}A}$ is 
the function $\alpha$ of Definition 2 while $\text{cd}_{\text{RI}B}$ is 
the function $\beta$ of Definition 2.  If $x < \text{SBC}(n,i)$, then $\alpha (n,i,x) = 0$,
and if $x \geq \text{SBC}(n,i+1)$, then $\alpha(n,i,x) = \binom{n}{i} = \text{SBC}(n,i+1)-\text{SBC}(n,i)$.
For $\text{SBC}(n,i) \leq x < \text{SBC}(n,i+1),\ \alpha(n,i,x) = x+1-\text{SBC}(n,i)$.
\item For $n > 0$,
$\alpha_{n,i} = \left\vert \text{I}A_{n,i}\right\vert = 
 \binom{n}{i} = \left\vert \text{I}B_{n,i}\right\vert = \beta_{n,i}$.
\item  If $a = a_{n,i,s}$, then $s = \alpha (n,i,a)$ and $a =\text{ERI}A(n,i,s)$; the
analogous statements hold, replacing $A$ with $B$, $\alpha$ with $\beta$ and
all occurrences of $a$ with $b$.
\item  For any system $X_{n,u}$ as in Definition 3, if $u \in U_n,\
a < b < 2^n$, then $\text{cd}_{\text{R}X}(n,u,b) - 
\text{cd}_{\text{R}X}(n,u,a) = \left\vert \left\{ x \in X_{n,u}\vert 
a < x \leq b\right\}\right\vert$.
\item Suppose that $\left\{ (n,u)\vert u \in U_n\right\}$ is P-TIME decidable.  
Then, if $\text{\textnormal{R}}X$ is tame, it is also P-TIME decidable.
\end{enumerate} 
\end{remark} 
\begin{proof}
Items 1. - 5. are obvious.  For item 6., note that for $n > 0,\ u \in U_n$ and 
$x \in \mathbb{Z}^+$ with $x < 2^n,\ \text{R}X(n,u,x)$ holds iff either $x = 
\text{cd}_{\text{R}X}(n,u,x) = 1$ or\\
\noindent
$\left (
 x > 1\text{ and }
\text{cd}_{\text{R}X}(n,u,x) = \text{cd}_{\text{R}X}(n,u,x-1) + 1\right )$.
\end{proof}
\begin{lemma}  Let $n > 0$.  Then:
\begin{enumerate}
\item  If $\pi$ is a permutation of 
$\left\{ 0, \ldots , 2^n-1\right\}$,
the admissibility of $\pi$  is equivalent to each of the 
following conditions:
\begin{enumerate} 
\item  for all $k < 2^n,\
\text{Weight}(\pi (k)) = \text{IStep}(n, k)$,
\item  for all $i < n,\ \pi\left [ \text{I}A_{n,i}\right ] =
\text{I}B_{n,i}$,
\item  $\text{I}S^*_n = \text{I}S_n \circ \pi$.
\end{enumerate}  
\item  There are $\prod_{i=0}^n \left (\binom{n}{i}!\right )$
admissible permutations of $\left\{ 0, \ldots , 2^n-1\right \}.$ 
\end{enumerate}
\end{lemma}
\begin{proof}  For 1., it is clear that (b) and (c) are each equivalent to (a),
so we argue that the admissibility of $\pi$ is equivalent to (a).  Let $\pi$
be any permutation of $\left\{ 0, \ldots , 2^n - 1\right\}$, let $k < 2^n$ and 
let $x \in D_{n,k},\ y \in D_{n, \pi (k)}$.  Then:
\[
2\text{IStep}(n,k) =  n + S^*_n(x)\text{  and  }
2\text{Weight}(\pi (k)) = n + S_n(y).  
\]
But the condition that the right hand sides of the last two displayed
equations are equal (for any $k$ and  any such $x, y$) defines the admissibility of $\pi$,
while (a) is the condition that the left hand sides are equal (for all $k$), and so the admissibility
of $\pi$ is equivalent to (a). 

For 2., note that an admissible permutation $\pi$ 
decomposes into the system of its restrictions to the $\text{I}A_{n,i}$.
Complete information 
about $\pi\upharpoonright \text{I}A_{n,i}$ is encoded by
the permutation, $\overline{\pi}_{n,i}$ of
$\left\{ 1, \ldots , \binom{n}{i}\right\}$ defined by:
\begin{equation}
\text{if } 1 \leq s \leq \binom{n}{i},\text{ then }
\overline{\pi}_{n,i}(s) :=\beta \left ( n,i,\pi \left ( a_{n,i,s}\right )\right ).
\end{equation}
Further, the $\overline{\pi}_{n,i}$
are arbitrary in the sense that if, for $i < n,\
\sigma_{n,i}$ is {\it any } permutation of
$\left\{ 1, \ldots , \binom{n}{i}\right\}$, then for each $n$
there is a (unique) admissible permutation $\pi_n$
of $\left\{ 0, \ldots , 2^n - 1\right\}$ such
that for each $i < n,\ \overline{\pi}_{n,i} = \sigma_{n,i}$.  Finally,
for fixed $n$, the product in 2. counts the number of such
systems $\left ( \sigma_{n,i}\vert i < n\right )$,
and so 2. follows.
\end{proof}
It is worth pointing out, here, that the proof of 2. of Lemma 1 provides
us with yet another view of admissible permutations, since they
correspond canonically to such systems of $\sigma_{n,i}$.
We will return to this view of admissible permutations in (3.5) below.
\subsection{Corollary 2 and Propositions 1, 2}
\label{subsec:3.2}
The next Corollary is an immediate consequence of Theorem 1, Corollary 1 and
2. of Lemma 1; it gives the existence of trim, strong triangular array for $\left\{ S^*_n\right\}$.
\begin{corollary}  For each $n$,
there are 
$\prod_{i=0}^n \left ( \binom{n}{i}!\right )$ representations 
\[
S^*_n =  \sum_{i=1}^nR^*_{n,i},
\]
where $\left ( R^*_{n,i}\vert\, 1 \leq i \leq n\right )$ is an
independent family of Rademacher random variables each of which 
has mean 0, variance 1 and depends 
only on $\mathbf{r}_n.$   Therefore, there exist (continuum many)
trim, strong triangular array representations of the sequence
$\left\{ S^*_n\right\}$.
\qed
\end{corollary}
The next Proposition builds on Corollary 2 to show that there are non-trim, strong 
triangular array representations of $\left\{ S^*_n\right\}$, by constructing one,
as a modification of a trim, strong one.  Thus, the existence of
trim strong triangular arrays for $\left\{S^*_n\right\}$ is not an immediate formal consequence 
of the existence of strong ones.
\begin{proposition}  There are non-trim, strong triangular arrays for $\left\{S^*_n\right\}$.
\end{proposition}
\begin{proof}
Fix a trim, strong triangular array, $\left\{R^*_{n,i}\right\}$, for 
$\left\{S^*_n\right\}$.  It is easy to see that 
for all sufficiently large $n$
we can find $i_1,\ i_2,\ k_1,\ k_2$
such that $i_1 \neq i_2,\  i_1,\ i_2 < n,\ k_1 < k_2 < 2^n$ and for
$j = 1, 2$:
\[
R^*_{n,i_1}(x) = (-1)^j\text{ for all }x \in D_{n,k_j}\text{ and }
R^*_{n,i_2}(x) = (-1)^{j-1}\text{ for all }x \in D_{n,k_j}.
\] 
Fixing a sufficiently large $n^*$ and then fixing such $i_1,i_2,j_1,j_2$, we define  $\left\{ R^{**}_{n,i}\right\}$
and then verify that it is a non-trim, strong
triangular array for $\left\{S^*_n\right\}$.  For $n \neq n^*$, or
$n = n^*$ and $i \neq i_1, i_2$, we let 
$R^{**}_{n,i} := R^*_{n,i}$.  For $j = 1,2$, we let:
\[
R^{**}_{n^*,i_1}(x) := (-1)^j\text{ for all }x \in D_{n^*+1,2k_1+j-1} 
\cup D_{n^*+1,2k_2+2-j}\text{ and }
\]
\[ 
R^{**}_{n^*,i_1}(x) := (-1)^{j-1}\text{ for all }x \in D_{n^*+1,2k_1+2-j} 
\cup D_{n^*+1,2k_2+j-1},
\]
\[
R^{**}_{n^*,i_2}(x) := (-1)^j\text{ for all }x \in D_{n^*+1,2k_1+2-j} 
\cup D_{n^*+1,2k_2+j-1}\text{ and } 
\]
\[
R^{**}_{n^*,i_2}(x) := (-1)^{j-1}\text{ for all }x \in D_{n^*+1,2k_1+j-1} 
\cup D_{n^*+1,2k_2+2-j}.
\]
Clearly $R^{**}_{n^*,i_1},\ R^{**}_{n^*,i_2}$ are 
Rademacher, with mean 0 and variance 1 and clearly they depend on $\varepsilon_{n^*+1}$.
By construction, we have guaranteed that for all $x \in D_{n^*,k_1}
\cup D_{n^*,k_2}$ we will have that
\[
R^{**}_{n^*,i_1}(x) + R^{**}_{n^*,i_2}(x) = R^*_{n^*,i_1)}(x) + R^*_{n^*,i_2}(x)\text{ and }
\]
\[
R^{**}_{n^*,i_1}(x) \cdot  R^{**}_{n^*,i_2}(x) = R^*_{n^*,i_1}(x) \cdot R^*_{n^*,i_2}(x).
\]
This suffices to show that the $R^{**}_{n^*,i}$ are independent, and that they
sum to $S^*_{n^*}$.  Thus, $\left\{ R^{**}\right\}$
is a strong, non-trim
triangular array for $\left\{S^*_n\right\}$, as required.
\end{proof}
The next Proposition is a ``non-persistence'' result,
showing that in any sequence, $\left\{\pi_n\right\}$, of
admissible permutations, there is no $n$ such that $\pi_{n+1}$
extends $\pi_n$, and that, in any trim, strong triangular
array representation, $\left\{R^*_{n,i}\right\}$, of $\left\{S^*_n\right\}$,
for any $n$, there is some $i$ such that $R^*_{n+1,i} \neq R^*_{n,i}$.
This is in contrast to the situation
for the $R_{n,i}$, so, as noted at the end of (1.2.2),
Proposition 2 begins to answer to Question 4.  Item 1. of
Proposition 2 appeals to 1. of Lemma 1.
\begin{proposition} 

\text{ }

\begin{enumerate}
\item  If $\left \{ \pi_n\right\}$ is any sequence of admissible permutations,
then for all $n,\ \pi_n \not\subseteq \pi_{n+1}$ .

\item  If $\left\{  R^*_{n,i}\right \}$  is a trim,
strong triangular array for 
$\left\{ S^*_n\right\}$, then for all $n$, there is $1 \leq i \leq n$ such
that $R^*_{n+1,i} \neq R^*_{n,i}$.
\end{enumerate}
\end{proposition}
\begin{proof}
For item 1., 
the most obvious obstacle to having $\pi_n \subseteq \pi_{n+1}$
is that there will be $k < 2^n$ such that 
$\text{IStep}(n+1, k) < \text{IStep}(n, k)$.
Since $\text{Weight}\left ( \pi_i(k)\right ) = \text{IStep}(i, k)$, for
$i = n,\ n+1$, clearly $\text{Weight}\left (\pi_{n+1}(k)\right ) < 
\text{Weight}\left ( \pi_n(k )\right )$, and therefore
$\pi_{n+1}(k) \neq \pi_n(k)$, for any such $k$.  

For item 2., note first that 
\[
\text{for any } k < 2^n,\ D_{n,k} = D_{n+1,2k} \cup D_{n+1,2k+1}.
\]
Now, let $\left\{ R^*_{n,i}\right\}$ be a trim, strong triangular array for
$\left\{S^*_n\right\}$, and $\left\{ \pi_n\right\}$ be the associated
sequence of admissible permutations.  Fix $n$, and, towards a contradiction,
assume that $R^*_{n+1,i} = R^*_{n,i}$ for all $1 \leq i \leq n$.  
We first show that 
\[
\text{for all } k < 2^n \text{ and for } j = \pi_{n+1}(2k), \pi_{n+1}(2k+1),\ 
D_{n+1,j} \subseteq D_{n,\pi_n(k)}.
\]
We argue this for $j = \pi_{n+1}(2k)$.  The case $j = \pi_{n+1}(2k+1)$ is similar.
Note that
\[
\text{for all } 1 \leq i \leq n,\ y \in D_{n+1,j} \text{ and } x \in D_{n+1,2k},\
(-1)^{1+\varepsilon_i(y)} = R^*_{n+1,i}(x).  
\]
But any such $x$ is in
$D_{n,k}$, and, by hypothesis, $R^*_{n+1,i}(x) = R^*_{n,i}(x)$, so
\[
(-1)^{1+\varepsilon_i(y)} = R^*_{n+1,i}(x) = R^*_{n,i}(x) = (-1)^{1+\varepsilon_i(z)} \text{ for any }
z \in D_{n,\pi_n(k)},
\]
i.e., $\displaystyle{y \in D_{n,\pi_n(k)}}$.  It is then immediate that 
\[
D_{n,\pi_n(k)} = D_{n+1,\pi_{n+1}(2k)} \cup D_{n+1,\pi_{n+1}(2k+1)} \text{ and so: }
\]
\[
\left\{ \pi_{n+1}(2k),\ \pi_{n+1}(2k+1)\right \} = 
\left\{ 2\pi_n(k), 2\pi_n(k)+1\right\},
\]
which means that $\pi_{n+1}(2k),\ \pi_{n+1}(2k+1)$ have
opposite parity.  Now, however, choose $k > 0$ so that 
$\text{IStep}(n+1, 2k) = \text{IStep}(n+1, 2k+1) = 1$
and $\pi_{n+1}(2k),\ \pi_{n+1}(2k+1) > 1$.
Then  $\pi_{n+1}(2k),\ \pi_{n+1}(2k+1)$ both have
weight 1 and therefore, both are even, contradiction!
\end{proof}
\subsection{Sequences of Admissible Permutations and their 
Encodings, Proposition 3}
\label{subsubsec:3.3}
A sequence, $\left\{ \pi_n\right\}$ of permutations of
$\left\{ 0, \ldots , 2^n-1\right \}$ (admissible or not) is naturally encoded by 
the two-place function $\Pi : \mathbb{N}\times \mathbb{N}
\to \mathbb{N}$ defined by $\Pi (n, k) = \pi_n(k)$ for $n > 0$ and
$k < 2^n,\ \Pi (0,0) = 0$ and $\Pi (n, k) = 2^n$,
for $k \geq 2^n$ (including when $n = 0$ and $k > 0$).  
We will approach the question of the complexity of the 
sequence in terms of the complexity of its natural encoding.

The next Proposition gives our general lower complexity
bound statement for arbitrary sequences of admissible 
permutations.  We follow the general approach developed
in the first paragraph of (1.3.5).
\begin{proposition}  The complexity of
$2^n$ is a lower bound for the complexity of {\it any} sequence of admissible permutations, and thus for any trim, strong
triangular array for $\left\{S^*_n\right\}$. 
\end{proposition}
\begin{proof} Let $\left\{\pi_n\right\}$ be any
sequence of admissible permutations and let $\Pi$ be its
natural encoding.  A simple, explicit expression for $2^n$
in terms of the $n$ values, $\Pi (n,1),\ldots , \Pi(n,n)$, of $\Pi$,
uniformly in $n$ is provided by Equation (6), below.  Thus, by paragraph 1
 of (1.3.5) Equation (6) and its proof give the statement
of the proposition.
\begin{equation}
 \text{For all } n > 0,\
2^n = 1 + \sum_{i = 1}^n\Pi (n,i).
\end{equation}
Equation 6 is immediate, from the following observations.  First, $\text{I}A_{n,1} = \{ 1, \ldots , n\}$.
Second, $\text{I}B_{n,1}$ is the set of
powers of $2$ below $2^n$ (since these are the weight 1 positive integers below $2^n$).
Finally, $\pi_n\left [ \text{I}A_{n,1}\right ] = \text{I}B_{n,1}.$   
\end{proof}
\subsection{The sequence $\left\{F_n\right\}$, its natural encoding, $F$, and Theorem 2}
\label{subsec:3.4}
For the next Definition, recall
our invocation of Definition 3 in (3.1, where the 
$b_{n,i,s}$ and $\alpha (n,i,k)$ are defined).
\begin{definition}  For all $n > 0,\ F_n$ is the permutation
of $\left\{ 0, \ldots , 2^n-1\right \}$ defined as follows.
If $0 \leq k < 2^n$, let $i = \text{IStep}(n,k)$.  Then
\[
F_n(k) := b_{n,i,s},\text{ where } 
 s = \alpha (n,i,k).
\]
Then,  take 
$F:  \mathbb{N}\times \mathbb{N}
\to \mathbb{N}$ to be the natural encoding of the sequence
$\left\{ F_n\right\}$.  Also, for use in (3.5) and referring to the first sentence
of (3.3), we take $\text{Inv}F:  \mathbb{N}\times \mathbb{N}
\to \mathbb{N}$ to be the natural encoding of the sequence
$\left\{ F^{-1}_n\right\}$.
\end{definition}
\noindent
In view of 2. of Lemma 1, these $F_n$ are (obviously)
very natural admissible permutations of the $\left\{ 0, \ldots , 2^n-1\right \}$.
Recalling that, $\text{ERIB}$ is the enumerating function for $\text{RI}B$, viz. 
the invocation of Definition 3 immediately preceding Remark 3,
note that $F(n,k) = \text{ERI}B(n,i,s)$, where $s$ is as in Definition 4. 
\begin{remark}  Note that our definition of $F_n$ is equivalent
to stipulating that, in terms of the notation used in Equation (5),
for all $i \leq n,\ \overline{F_n}_i$ is the identity permutation of
$\left\{ 1, \ldots , \binom{n}{i}\right\}$.
Note, also, that with $i = \text{\textnormal{IStep}}(n,k)$ and $s =
\alpha(n,i,k)$, then, in fact, $s = k - \text{\textnormal{SBC}}(n,i)$; further,
for these $i, s$, we also have that $s = \beta (n,i,F(n,k))$.
\qed
\end{remark} 
\begin{proposition}   As functions of $(n,i)$, with $i \leq n$, the
binomial coefficients $\binom{n}{i}$ and \textnormal{SBC} are 
computable in time polynomial in $n$.  The function $\text{\textnormal{IStep}}$ is \textnormal{P-TIME}.  Further,
the relation $\text{\textnormal{RI}A}$ is \textnormal{tame}.
\end{proposition}
\begin{proof}
The algorithm for computing the binomial coefficents is simply
to generate the needed portion of Pascal's triangle
using the familiar addition identity.  This requires $O\left ( n^2\right )$
additions, and then the computation of $\text{SBC}(n,i)$, requires
$i$ more additions.  All of the summands remain below $2^n$, obviously.
For additional results on the computation of the binomial coefficients, see 
 {[}\ref{Beame}{]},  {[}\ref{Hesse}{]}.   That $\text{IStep}$ is P-TIME 
then follows immediately from item 1. of Remark 2; note that the ``search''
is bounded by $n$.
For the final statement, recall (item 2. of Remark 3) that the function $\alpha$
of Definition is $\text{cd}_{\text{RI}A}$, and let  
$x \in \mathbb{Z}^+$.  Note that if 
$x < \text{SBC}(n,i)$, then $\alpha (n,i,x) = 0$,
and if $x \geq \text{SBC}(n,i+1)-1$, then $\alpha (n,i,x) = \binom{n}{i}$.
For $\text{SBC}(n,i) \leq x < \text{SBC}(n,i+1)-1,\ \alpha (n,i,x) = 1+ x -
\text{SBC}(n,i)$, by Remark 4.  Thus, $\alpha$ is P-TIME.
\end{proof}
The next Lemma embodies, in abstract form (to facilitate multiple applications),
S. Buss's suggestion of combining tameness with a 
binary search argument to show that if $\text{R}X$ is tame
then $\text{ER}X$ is P-TIME.  Along with Item 6. of Remark 3, 
Lemma 2 represents the main apport of
the hypothesis of tameness.
\begin{lemma}  Suppose that $U_n$, the $X_{n,u}$, etc., are as in
Definition 3 and suppose that $\text{\textnormal{R}}X$ is \textnormal{tame}.  Then the enumerating function
$\text{\textnormal{ER}}X$ is \textnormal{P-TIME}. 
\end{lemma}
\begin{proof}
Fix $n,\ u \in U_n$ and $s$ with $1 \leq s \leq \left\vert X_{n,u}\right\vert$.
Let $\text{cd} = \text{cd}_{\text{R}X}$ be the cardinality function associated
with the $X_{n,u}$; by hypothesis, $\text{cd}$ is 
P-TIME.
Start from $a_0 = 0, b_0 = 2^{n}-1,\ s_0 = s$.  Having
defined $a_i,\ b_i,\ s_i$, we let $m_i := \lfloor \left ( a_i+b_i\right )/2\rfloor$ and we 
consider whether $ \text{cd}\left ( n,u,m_i\right ) - \text{cd}\left ( n,u,a_i\right ) \geq s_i$.
If so, we take $a_{i+1} = a_i,\ b_{i+1} = m_i,\ s_{i+1} = s_i$.
Otherwise, we take $a_{i+1} = m_i,\ b_{i+1} = b_i,\ s_{i+1} = 
\ s_i + \text{cd}_{\text{R}X}\left ( n,u,a_i\right ) - \text{cd}_{\text{R}X}\left ( n,u,m_i\right )$.  
Then, clearly, for some $k \leq n$ we will have $s_k = 1,\ a_k = b_k-1$ and $b_k = x_{n,u,s}$.
Thus, $\text{ER}X$ is P-TIME.
\end{proof}
S. Buss also sketched for us an argument that became the proof of the next Lemma. 
\begin{lemma}
\text{\textnormal{(S. Buss)}}  $\text{\textnormal{RI}}B$ is \textnormal{tame}. 
\end{lemma}
\begin{proof}  We show that the function $\beta$ of
Definition 2 is P-TIME.  This suffices, since by (2) of Remark 3, $\beta$ is $\text{cd}_{\text{RI}B}$.
Let $j \leq n$ and let $b \in \mathbb{N}$ with $b < 2^n$.
Without loss of generality we may assume $0 < j,b$ and $j < n$.
Let $\ell = \text{min}(j,\text{Weight}(b))$,
so $\ell \geq 1$.  Let $i_1 > \ldots > i_\ell$ be the $\ell$ largest
$i\text{'s}$ such that 
the $i^{\text{th}}$ bit in the binary expansion of $b$ is 1. 
Clearly $\ell$ and the $i_s$ are computed in time polynomial in $n$.

If $j > i_1$, then $\beta(n,j,b) = 0$, so assume
that $j \leq i_1$.  If $j = 1$, then clearly $\beta (n,1,b) = i_1$,
so assume that $j > 1$ and so $i_1 > 1$.  
If $x \in \mathbb{N}$ with $x \leq b$, then either $x = b$
or there is unique
$s$ with $1 \leq s \leq \ell$ such that the $s\text{-th}$ bit
in the binary expansion of $x$ is 0 but for all $1 \leq t < s$,
the $t\text{-th}$ bit in the binary expansion of $x$ is 1.
 
Note that if $1 \leq s \leq \ell$, then 
$\binom{i_s-1}{j+1-s}$ counts the number
of such $x < b$ with $\text{Weight}(x) = j$.  Finally, this means that
if $\text{Weight}(b) = j$, then $\beta(n,j,b) = 1+\sum_{s=1}^\ell\binom{i_s-1}{j+1-s}$,
while otherwise,   $\beta(n,j,b) = \sum_{s=1}^\ell\binom{i_s-1}{j+1-s}$.
\end{proof}
We record a few observations related to the proof of Lemma 3 that
will be useful in the proof of Theorem 2.
First, note that, for $i_s = 1$,  the binomial coefficient
$\binom{i_s-1}{j+1-s}$ is just 1; for $i_s > 1$, the coefficients
that occur in the final paragraph
of the proof can be expressed in terms of SBC:  
\begin{equation*}
\binom{i_s-1}{j+1-s} = \textnormal{SBC}\left (i_s-1, j+2-s\right ) - \textnormal{SBC}\left (i_s-1, j+1-s\right ).
\end{equation*}
Thus, the function $\beta$ has a simple expression in terms of SBC.
Next, note that $\chi_{\text{RI}B}(n,j,x) = 1$ iff ($j = x = 1$ or ($x > 1$ and $\beta (n,j,x) = \beta (n,j,x-1)+1$)).
Thus, $\chi_{\text{RI}B}$ also has a simple expression in terms of SBC, since $\beta$ does.   This is similar to the
argument for item 6. of Remark 3.

Since $\text{Weight}$ does not depend on $n$, we can
naturally ``put together'' the different branches, indexed by $n$, of the
$\text{ERI}B$ function into a single enumerating function, $\text{EW}$,
for $\text{Weight}$; this is Definition 5.  The final assertion of
Corollary 3 is an easy consequence of Lemmas 2, 3:  $\text{EW}$ is P-TIME.
This result is used in the proof of Theorem 2 and is also of some interest in its own
right, since the sequence
http://oeis.org, 2010, Sequence A066884, {[}\ref{OEIS}{]}, 
encodes $\text{EW}$; {[}\ref{OEIS}{]} does not indicate
that this sequence is P-TIME and gives no closed
form. 
\begin{definition}
For $j, t \in \mathbb{N}$:
\begin{equation*}
\text{EW}(j,t) := \begin{cases}
                            0, &\text{ if } j = 0\\
                            \text{ the }t^{\text{th}} m\text{ such that Weight}(m) = j & \text{ if } j > 0
                         \end{cases}
\end{equation*}
\end{definition}
\begin{corollary}  The relations $\text{\textnormal{RI}}A,\ 
\text{\textnormal{RI}}B$ are \textnormal{P-TIME decidable}.  
The function $\text{EW}$ is \textnormal{P-TIME}.
The relation between $n, k, m$
expressed by the Equation 7, which follows, is also \textnormal{P-TIME decidable}.
Given $n$ and $k < 2^n$, this equation has a unique solution, $m$,  which, as a function
of $(n,k)$, is also \textnormal{P-TIME}.
\begin{equation}
\beta (n,\text{\textnormal{IStep}}(n,k),m)\cdot
\chi_{\text{RI}B}(n,\text{\textnormal{IStep}}(n,k),m) =
\alpha (n,\text{\textnormal{IStep}}(n,k),k).
\end{equation}
\end{corollary}
\begin{proof}  By Proposition 4, $\text{RI}A$ is tame, and by Lemma 3, so is
$\text{RI}B$.  Since the hypothesis of item 6. of Remark 3 clearly holds
for $\text{RI}A,\ \text{RI}B$, these relations are P-TIME decidable. 

For $\text{EW}$, note that if $j > 0$, then for any $t$, taking $n = \text{max}(j,t)+1$,
we will have $t \leq \binom{n}{j}$.  Therefore, 
$\text{EW}(j,t) < 2^n$, and so $\text{EW}(j,t) = \text{ERI}B(n,j,t)$. 
Since $\text{RI}B$ is P-TIME
decidable, $\chi_{\text{RI}B}$ is P-TIME, and so  
the third sentence of the Corollary is immediate from Lemma 3 and
the first sentence.  

Let $j = \text{IStep}(n,k)$ and 
note that $\alpha(n,j,k) > 0$.  This is the point
of multiplying by $\chi_{\text{RI}B}(n,j,m)$: 
to ensure that $m \in \text{I}B_{n,j}$.
The unique solution $m$ is computed
as $\text{EW}(j,k+1-\text{SBC}(n,j))$. 
\end{proof}
The observations in the last paragraph of the proof of Corollary 3
will be used in the proof of Theorem 2.
Of course the P-TIME decidability of $\text{RI}A$ can be established
quite simply and directly from Proposition 4, but the approach taken is more efficient.
\begin{theorem}  $F$ is \textnormal{P-TIME and simply computed in terms of SBC}.
\end{theorem}
\begin{proof}
We first argue for the lower complexity bound much as in Proposition 3, but
with a slightly cleaner expression for $2^n$ given by Equation (8), below, rather
than by Equation (6).  We then show that $F$ is P-TIME; the upper
complexity bound then follows by the general argument given in (1.3.5).  We note: 
\begin{equation}
2^n = F(n, n)+F(n,n).
\end{equation}
For Equation (8), the relevant observations are
that $n$ is the largest element of $\text{I}A_{n,1},\ 2^{n-1}$ is the largest
element of $\text{I}B_{n,1}$ and that for all $n, i,\ F_n$ maps $\text{I}A_{n,i}$ onto
$\text{I}B_{n,i}$ in order-preserving fashion.  

The rest of the proof follows the strategy laid out in (1.3.5).  To see that
$F$ is P-TIME we will argue that
\begin{equation} 
\text{For }n > 0\text{ and } k < 2^n,\
m = F(n,k) < 2^n\text{ is the unique solution of Equation (7).}
\end{equation}
\noindent
This is immediate from the last observation given in connection with Equation (8) and 
clearly suffices to show that $F$ is P-TIME, in view of Lemmas 2, 3 and Corollary 3.
That $F$ is simply computed in terms of SBC follows from Equation (7) since all of
the functions that figure there are simply computed in terms of SBC. 
For IStep, this is by item 1. of Remark 2.   For $\alpha$,
this is by item 2. of Remark 3.    For $\beta$ and
$\chi_{\text{RI}B}$ this is by the first paragraph following the proof of Corollary 3.
\end{proof}
\subsection{Obtaining SBC and related Questions}
\label{subsec:3.5}
We begin by showing how to obtain SBC from $F$ and $\text{Inv}F$ (the latter was also introduced in Definition 4).
\begin{corollary}  \textnormal{SBC} is simply computed in terms of $F$ and $\text{Inv}F$.  $\textnormal{Inv}F$
is simply computed in terms of \textnormal{SBC}.  Thus, the joint complexity of $F$ and $\textnormal{Inv}F$
is exactly that of \textnormal{SBC}.
\end{corollary}
\begin{proof}  For the first assertion, recall that $\text{SBC}(n,0) = 0$ and note that $\text{SBC}(n,n+1) = 2^n$.   
For $1 \leq i \leq n$, note that $\text{SBC}(n,i) = \text{Inv}F\left ( n,2^i-1\right )$.  But $2^i-1 =
2F(i,i)-1$.

That $\text{Inv}F$ is simply computable from SBC follows from the material of (3.4), and in particular from the following
``dual version'' (interchanging Step and Weight) of Equation 7:
\begin{equation*}
\alpha (n,\text{\textnormal{Weight}}(m),k)\cdot
\chi_{\text{RI}A}(n,\text{\textnormal{Weight}}(n,m),k) =
\beta (n,\text{\textnormal{Weight}}(m),m).
\end{equation*}
\noindent
As in (3.4), all of the functions in the previous displayed equation are simply computed from SBC, and,
given $(n,m)$, the unique solution, $k$, is computed as $\text{SBC}(n,\text{Weight}(m),t)$, where
$m = \text{EW}(\text{Weight}(m),t)$.  We can easily compute $t$ from SBC (and EW) using a binary search argument
with initial interval $\left [ 1,\binom{n}{\text{Weight}(m)}\right ]$.  But this unique solution is just $\text{Inv}F(n,m)$.
The final assertion is immediate from the first two.
\end{proof}
\begin{remark}  It would be ideal if we could show that SBC is simply computed from $F$ alone, since then
the complexity of $F$ would be exactly that of SBC.  The specific obstacle
is being able to carry out the calculation of $t$ in the last sentence of the proof of the Corollary in terms
of $F$ alone, without the use of the binomial coefficient or the function EW.  An indication that this
obstacle may be serious is the general phenomenon that an inverse of a function, $f$, can be significantly more
complex than $f$ itself.  Symmetrically, if we could eliminate the use of $F$ to compute $2^i-1$, it would follow
that the complexity of $\text{Inv}F$ is exactly that of SBC.  This
seems somewhat more feasible.  
\end{remark}
We conclude this section by tieing up some ``odds and ends''.  We first argue 
that $\left\{ F_n\right\}$ is the simplest sequence
of admissible permutations, and so its corresponding triangular array representation 
is the simplest trim, strong triangular array representation for
$\left\{ S^*_n\right\}$.  Finally, we make some
observations concerning the contrast between Equations (6) and (8) 
in light of this status of $\left\{ F_n\right\}$.

Remark 4 and the proof of item 2. of Lemma 1 are the main elements of our argument 
that $\left\{ F_n\right\}$ is the simplest sequence
of admissible permutations.  Recall that each $F_n\upharpoonright IA_{n,i}$ is
the order-preserving bijection between $ IA_{n,i}$ and $ IB_{n,i}$.
While the claim that this is the simplest bijection between these
sets may not be entirely clear, it is far clearer that the identity
permutation on $\left\{ 1, \ldots , \binom{n}{i}\right\}$
{\it is } the simplest permutation of this set.  By Remark 4,
each $\overline{F}_{n,i}$ is the identity permutation on 
$\left\{ 1, \ldots , \binom{n}{i}\right\}$ and so from the 
point of view of the proof of Lemma 1 and the subsequent
paragraph, $\left\{ F_n\right\}$ really is simplest, since
it is represented by the system where each $\sigma_{n,i}$
is the identity permutation.  It should, however, be acknowledged
that we have ``built in'' the role of the increasing enumerations
in this way of representing admissible permutations.

Regarding the contrast between Equations (6) and (8),
what is really at issue is to be able to easily
identify $k(n) = \left ( \pi_{n+1}\right )^{-1}\left ( 2^n\right )$
as a function of $n$, since, trivially, we'll always have that
$2^n = \pi_{n+1}(k(n))$.  For $\left\{ F_n\right\}$, Equation
(8) is based on the easy identification of $k(n)$ as simply
being $n+1$.  It is then natural to expect that for more
complex sequences $\left\{ \pi_n\right\}$, the corresponding
function $k(n)$ will also be more complex, leaving us 
only Equation (6) rather than a simple analogue of Equation (8).  This has
some features in common with the issues discussed in Remark 5, above.
\section{Construction of the Variants of $F$ and Theorem 3}
\label{sec:4}
\subsection{The functions $G$ and $H$}
\label{subsec:4.1}
Here we
construct the variants, $G$ (Definition 7) and $H$ (Definition 11), of $F$.
We impose additional requirements on the admissible permutations, 
$G_n$ and $H_n$, respectively, that are to be encoded.  

The motivation for introducing these variants is to obtain sequences of
permutations whose natural encodings are still P-TIME, 
with SBC as an upper complexity bound (this
will be the content of Theorem 3) and
where the orbit structures of the individual permutations are simpler than those 
of the $F_n$.  We construct the $G_n$ so as to
maximize the number of fixed points.  The construction of the $H_n$ goes
farther:  once all possible fixed points have been identified (the same ones
as for the $G_n$), we maximize the number of two-cycles, so that the $H_n$ 
are as close as possible to being self-inverse.

We construct the $G_n$ in two stages:  we first note the fixed points are
the elements of the $\text{I}A_{n,i} \cap \text{I}B_{n,i}$.
We then proceed much as for $F$:  for each $1 \leq i \leq n$, map ``what
is left of'' $\text{I}A_{n,i}$ in order preserving fashion onto ``what is left of''  
$\text{I}B_{n,i}$.  Of course, this requires that these two sets have the same cardinality;
this will be obvious for the construction of $G_n$, as noted in Remark 6.  

We construct the $H_n$ in three stages, with
the first stage being identical to the first stage in the construction of the $G_n$.
We interpolate a new second stage, where we identify a maximal set of two-cycles.
The third and final stage is analogous to the second stage in the definition of the
$G_n$, in that, for each $1 \leq i \leq n$, we map ``what is left of'' $\text{I}A_{n,i}$ 
 in order preserving fashion onto ``what is left of'' $\text{I}B_{n,i}$.  This time, ``what is left''
means after removing the fixed points {\it and } the points involved in the two-cycles
identified in the second stage.  As in the second stage of the construction of $G_n$,
in order to carry out the third and final
stage for the $H_n$, it must again be true that for each $i$, ``what is left of'' $\text{I}A_{n,i}$ has the same
cardinality as ``what is left of'' $\text{I}B_{n,i}$.  This is the content of Proposition 5.

The constructions of the $G_n$ and of the $H_n$ will both be uniform in $n$, so
for the remainder of this subsection, we take $n$ to be fixed.
The next definition is analogous to Definition 2.  It
introduces the $\text{I}A^1_{n,i}$
and $\text{I}B^1_{n,i}$: ``what is left of $\text{I}A_{n,i}$, resp. $\text{I}B_{n,i}$'', after removing
the fixed points.
\begin{definition}  For $i < n,\ 
\text{I}A^1_{n,i} := \text{I}A_{n,i} \setminus \left ( \text{I}A_{n,i} \cap \text{I}B_{n,i} \right ),\
\text{I}B^1_{n,i} := \text{I}B_{n,i} \setminus \left ( \text{I}A_{n,i} \cap \text{I}B_{n,i} \right ).$
We also let $\overline \alpha^0_{n,i} = \overline \beta^0_{n,i} :=
\left\vert \text{I}A_{n,i} \cap \text{I}B_{n,i}\right\vert$, and set 
$\alpha^1_{n,i} :=\left\vert\text{I}A^1_{n,i}\right\vert,\ 
\beta^1_{n,i} := \left\vert\text{I}B^1_{n,i}\right\vert$.  

{\bf We now invoke Definition 3 with $d = 1$ and  $U_n = \{ 0, \ldots , n\}$
and with }$X_{n,i} = \text{I}A^1_{n,i}${\bf\ or }$X_{n,i} = \text{I}B^1_{n,i}$.
{\it This defines }$\text{RI}A^1,\ \text{RI}B^1,\ \text{cd}_{\text{RI}A^1},\ 
\text{cd}_{\text{RI}B^1},\ \text{ERI}A^1,\ \text{ERI}B^1$.
We use $\alpha^1,\ \beta^1$ to denote 
$\text{cd}_{\text{RI}A^1},\ 
\text{cd}_{\text{RI}B^1}$, respectively.
{\it The notation for the increasing enumerations will be }$a^1_{n,i,s},\ b^1_{n,i,s}$.
\end{definition}
\begin{remark}  For $i < n$, the following observations are obvious:
\begin{enumerate}  
\item $\alpha^1_{n,i} = \binom{n}{i} - \overline \alpha^0_{n,i} =
\binom{n}{i} - \overline \beta^0_{n,i} = \beta^1_{n,i}$, 
\item for $k \in \text{I}A_{n,i},\ k \in \text{I}A^1_{n,i}$
iff $\text{Weight}(k) \neq i$; for
$m \in \text{I}B_{n,i},\ m \in \text{I}B^1_{n,i}$ iff $\text{IStep}(n,k) \neq i$.
\qed
\end{enumerate}
\end{remark}
\begin{definition}  $G_n$ is the permutation of $\left\{ 0, \ldots , 2^n-1\right \}$ defined as follows.
If $0 \leq k < 2^n$, let $i = \text{IStep}(n,k)$, then:

\begin{equation}
    G_n(k) := \begin{cases}
               k               & \text{if }\text{Weight}(k) = i \\
               b^1_{n,i,s} & \text{where } s = \alpha^1(n,i,k),\text{ otherwise.}
           \end{cases}
\end{equation}
Take 
$G:  \mathbb{N}\times \mathbb{N}
\to \mathbb{N}$ to be the natural encoding of the sequence
$\left\{ G_n\right\}$. 
\end{definition} 

It is clear that $G_n$ is an admissible permutation of 
$\left\{ 0, \ldots , 2^n-1\right \}$, with the additional
property that $G_n$ is the identity on $k$ such that
$\text{IStep}(n,k) = \text{Weight}(k)$, i.e., $G$
is maximal, among admissible permutations of 
$\left\{ 0, \ldots , 2^n-1\right \}$, for agreement
with the identity permutation.  Further, in analogy with
Remark 4 and (3.5), we argue that $G$ is the simplest such 
admissible permutation.  This is based on the analogues
of the $\overline{F}_{n,i}$ defined on $\left\{ 1, \ldots ,
\alpha^1_{n,i}\right\}$, starting from $G_n\upharpoonright A^1_{n,i}$.
Each of these is the identity on $\left\{ 1, \ldots ,
\alpha^1_{n,i}\right\}$.

We turn now to the definition of the $H_n$ and $H$. 
Once again, we will proceed in analogy with Defintions 2 - 4.  Our first
task is to identify those $k$ which will be part of a two-cycle.
 
In order to motivate what follows, suppose that $\pi$
is an admissible permutation of $\left\{ 0, \ldots , 2^n-1\right \}$
with the property we have built into $G_n$:  that $\pi(s) = s$
whenever $\text{IStep}(n,s) = \text{Weight}(s)$.  Suppose
further that $k \neq m, \pi (k) = m$ and $\pi( m) = k$.  Let
$i = \text{IStep}(n,k),\ j = \text{Weight}(k)$.  Then
$i \neq j,\ i = \text{Weight}(m),\
j = \text{IStep}(n,m)$.  Stated otherwise, we have
that $k \in \text{I}A^1_{n,i} \cap \text{I}B^1_{n,j}$ and
$m \in \text{I}A^1_{n,j} \cap \text{I}B^1_{n,i}$.

\begin{definition}
For  $i, j < n$ with $i \neq j$ we set:  $\text{I}C^1_{n,i,j} := \text{I}A^1_{n,i} \cap \text{I}B^1_{n,j}$, and
$\gamma^1_{n,i,j} := \left\vert \text{I}C^1_{n,i,j}\right\vert$.  

{\bf We now invoke Definition 3 with $d = 2$ and  $U_n = \{ (i,j) \vert 0 \leq i, j \leq n,\ i \neq j\}$
and with }$X_{n,i,j} = \text{I}C^1_{n,i,j}$.  
{\it This defines }$\text{RI}C^1,\ \text{cd}_{\text{RI}C^1},\ \text{ERI}C^1$. 
We use $\gamma^1$ to denote $\text{cd}_{\text{RI}C^1}$.
{\it The notation for the increasing enumerations will be }$c^1_{n,i,j,s}$.
\end{definition}
\begin{remark}  The following observations are obvious:

For
 $\displaystyle{ i < n,\  
\text{I}A^1_{n,i} =
 \bigsqcup_{0 \leq j < n,\ j \neq i}
 \text{I}C^1_{n,i,j}}$, 
 and
 $\displaystyle{\text{I}B^1_{n,i} =
 \bigsqcup_{0 \leq j < n,\ j \neq i}
 \text{I}C^1_{n,j,i}}$.
\qed
\end{remark}
It would be natural to attempt to match up the elements of
the $\text{I}C^1_{n,i,j}$ with those of correspoding $\text{I}C^1_{n,j,i}$ 
to form the two-cycles.  However, the following example
shows that even for fairly small $n$, this will not be possible,
since it can happen that for certain $i \neq j,\ i,\ j < n,\ 
\gamma^1_{n,i,j} \neq \gamma^1_{n,j,i}$.
When $n=8$, we have: 
\[
\text{I}C^1_{8,2,4} = \{ 15,23,27,29,30\} \text{, while }
\text{I}C^1_{8,4,2} = \{ 96,129,130,132,136,144,160\}.
\]
When $\gamma^1_{n,i,j} > 
\gamma^1_{n,j,i}$ there are various reasonable ways of
choosing the $\gamma^1_{n,j,i}-\text{many}$ elements
of $\text{I}C^1_{n,i,j}$ which will form 2-cycles with the
elements of $\text{I}C^1_{n,j,i}$.  The particular way we
have chosen in what follows is to {\it exclude} the ``extreme''
elements of $\text{I}C^1_{n,i,j}$:  those that are farthest from
the elements of $\text{I}C^1_{n,j,i}$.  This is codified in the next Definition.   
\begin{definition}   For $i, j < n$ with $i \neq j$, let 
$\overline{\gamma}^1_{n,i,j} := 
\text{min}\left ( \gamma^1_{n,i,j}, \gamma^1_{n,j,i}\right )$ and set:
\begin{equation}
    \text{I}\overline{C}^1_{n,i,j} := \begin{cases}
                \text{I}C^1_{n,i,j}               & \text{if }\gamma^1_{n,i,j} \leq \gamma^1_{n,j,i} \\
              \left\{ c^1_{n,i,j,s}\vert 1 \leq s \leq \gamma^1_{n,j,i}\right\}
                                                              & \text{if }\gamma^1_{n,i,j} > \gamma^1_{n,j,i}\text{ and }i > j \\
              \left\{ c^1_{n,i,j,t+s}\vert 1 \leq s \leq \gamma^1_{n,j,i}\right\}
                                                               & \text{where }t = \gamma^1_{n,i,j} - \gamma^1_{n,j,i}\text{ otherwise.}                                                             
      \end{cases}
\end{equation}
In the second or third case, let
$\text{I}C^2_{n,i,j} :=
\text{I}C^1_{n,i,j} \setminus \text{I}\overline{C}^1_{n,i,j}$.  

{\bf We now invoke Definition 3 with $d = 2$ and  $U_n = \{ (i,j) \vert 0 \leq i, j \leq n,\ i \neq j\}$
and with }$X_{n,i,j} = \text{I}\overline{C}^1_{n,i,j}$.
{\it This defines }$\text{RI}\overline{C}^1,\ \text{cd}_{\text{RI}\overline{C}^1},\ \text{ERI}\overline{C}^1$.
We use $\overline{\gamma}^1$ to denote $\text{cd}_{\text{RI}\overline{C}^1}$.
{\it The notation for the increasing enumerations will be }$\overline{c}^1_{n,i,j,s}$.
\end{definition}
\begin{definition}
For $i < n$, we set:
\[
\text{I}A^2_{n,i} := \bigsqcup_{0 \leq i, j < n,\ i \neq j}\text{I}C^2_{n,i,j}\text{ and }
\text{I}B^2_{n,i} := \bigsqcup_{0 \leq i, j < n,\ i \neq j}\text{I}C^2_{n,j,i}.
\]
We also set $\alpha^2_{n,i} := \left\vert\text{I}A^2_{n,i}\right\vert$ and
$\beta^2_{n,i} := \left\vert\text{I}B^2_{n,i}\right\vert$.

{\bf We now invoke Definition 3 with $d = 1$ and  $U_n = \{ 0, \ldots , n\}$
and with }$X_{n,i} = \text{I}A^2_{n,i}${\bf or }$X_{n,i} = \text{I}B^2_{n,i}$.
{\it This defines }$\text{RI}A^2,\ \text{RI}B^2,\ \text{cd}_{\text{RI}A^2},\ 
\text{cd}_{\text{RI}B^2},\ \text{ERI}A^2,\ \text{ERI}B^2$.
We use $\alpha^2,\ \beta^2$ to denote 
$\text{cd}_{\text{RI}A^2},\ 
\text{cd}_{\text{RI}B^2}$, respectively.
{\it The notation for the increasing enumerations will be }$a^1_{n,i,s},\ b^1_{n,i,s}$.
\end{definition}
\begin{proposition}
For $i < n,\ \alpha^2_{n,i} = \beta^2_{n,i}$ .
\end{proposition}
\begin{proof}
 We note first that 
\[
\text{I}A^2_{n,i} = \text{I}A^1_{n,i} \setminus\ 
 \bigsqcup_{0 \leq j < n,\ j \neq i}
\text{I} \overline{C}^1_{n,i,j}\text{ and that }
\text{I}B^2_{n,i} = \text{I}B^1_{n,i} \setminus\ 
\bigsqcup_{0 \leq j < n,\ j \neq i}
 \text{I}\overline{C}^1_{n,j,i}.  
\]
This follows from 
 the definitions of the $\text{I}A^2_{n,i},\ \text{I}B^2_{n,i}$ and the $\text{I}\overline{C}^1_{n,i,j}$,
 and Remark 7.
 But then, since, by construction, $\overline{\gamma}^1_{n,i,j} 
= \overline{\gamma}^1_{n,j,i}$, for all relevant $n,i,j$, we have that
\[
\left\vert \bigsqcup_{0 \leq j < n,\ j \neq i}
 \text{I}\overline{C}^1_{n,i,j}\right\vert
  = \left\vert \bigsqcup_{0 \leq j < n,\ j \neq i}
 \text{I}\overline{C}^1_{n,j,i}\right\vert.
\]
 Finally, by construction,  $\left\vert A^1_{n,i}\right\vert =
 \left\vert B^1_{n,i}\right\vert$.
It then clearly follows that $\alpha^2_{n,i} = 
\beta^2_{n,i}.$
\end{proof}
We can now complete the construction of the $H_n$.
Proposition 5 makes it clear that
the third case of Equation (12), below, will provide a coherent
definition and that the $H_n$ we define there are admissible
permutations of $\left\{ 0, \ldots , 2^n-1\right \}$ with the
additional property of $G_n$.
\begin{definition}
$H_n$ is the 
permutation of $\left\{ 0, \ldots , 2^n-1\right \}$ 
 defined as follows.
If $0 \leq k < 2^n$, let $i = \text{IStep}(n,k), j = \text{Weight}(k)$.  Then:

\begin{equation}
    H_n(k) := \begin{cases}
               k               & \text{if }j = i \\
               \overline{c}^1_{n,i,j,s}   & \text{where }s = \overline{\gamma}^1_n(i,j,k)
                                                            \text{ if } k \in \text{I}\overline{C}^1_{n,i,j} \\
               b^2_{n,i,s}                       & \text{where } s = \alpha^2_n(i,k),\text{ otherwise.}
                      \end{cases}
\end{equation}
Also (as usual), let:
$H:  \mathbb{N}\times \mathbb{N}
\to \mathbb{N}$ to be the natural encoding of the sequence
$\left\{H_n\right\}$ of admissible
permutations. 
\end{definition}

With reference to the discussion in the final
paragraph of (3.5), related to the form
of the lower bound expression, we should note here,
that even for $G$, the situation is somewhat more
complicated:  it may fail to be true that $G_n(n)
= 2^{n-1}$ (or, in the notation of the final paragraph of (3.5),
that $k(n) = n$):  this will happen exactly if $n$ is
is a power of 2, since then $\text{Weight}(n) = 1 =
\text{IStep}(n,n)$ and so $G_n(n) = n \neq 2^n$.  In
this case, however, we'll have that $G_n(n-1) = 2^{n-1}$.  This is the basis
for Equation (13), below, which is the analogue for $G$ of Equation (8).
We have not carried out a similar analysis of the $k(n)$ function for $H$, 
and so, in the proof of Theorem 3, in (4.2), we content ourselves
with the general lower bound expression given by Equation (6).
Similar issues represent similar (but even worse) obstacles to obtaining
an analogue of Corollary 4 for $G$ or $H$.
\subsection{Theorem 3}
\label{subsec:4.2}
The analogue of Theorem 2 for the functions $G$ and $H$
is provided by Theorem 3.  
The lower bound statement argument
divides, as just discussed, but things rejoin for the 
proof that $G$ and $H$ are P-TIME.  Proposition 6 is the
main technical tool; it incorporates the contributions of Proposition 4, Lemma 2,  Lemma 3 
and Corollary 3 in the proof of Theorem 2.
\begin{proposition}  Each of the following relations is both \textnormal{P-TIME decidable and tame}:\\
$\text{\textnormal{RI}}A^1,\ \text{\textnormal{RI}}A^1,\ \text{\textnormal{RI}}C^1,\
\text{\textnormal{RI}}\overline{C}^1,\ \text{\textnormal{RI}}A^2,\
\text{\textnormal{RI}}B^2$.  Also, all of the case conditions of
Equations (10) and (12) are \textnormal{P-TIME decidable}. 
\end{proposition}
\begin{proof}
For each of the listed relations, the hypothesis of item 6. of
Remark 3 clearly holds, and so it will suffice to establish tameness.
We weave our way through the statements to be proved in the following order.
First, the P-TIME decidability of the case condition of Equation (10), then 
the tameness of $\text{RI}A^1,\ \text{RI}B^1,\ 
\text{RI}C^1$, then the P-TIME decidability of the second case condition of Equation (12) and finally,
the tameness of $\text{RI}\overline{C}^1,\ \text{RI}A^2$ and $\text{RI}B^2$.
Throughout the proof we will have $x \in \mathbb{Z}^+$ with $x < 2^n$.

For the case condition of Equation (10) (and the first case condition of Equation (12)), 
our starting point is Proposition 4, itself.  The case condition 
 is just whether $
\text{IStep}(n,k) =
\text{Weight}(k)$:
if so, then it is the first case of Equations (10), (12) of Definitions 7, 11 that applies:
$G(n,k) = H(n,k) = k$.  It follows from Proposition 4 
that the relation expressed by the last displayed equation
is P-TIME decidable.

For the tameness of $\text{RI}A^1$, note that
for each $(n,i),\ \text{I}A^1_{n,i} =\text{I}A_{n,i} \setminus 
\left ( \text{I}A_{n,i} \cap \text{I}B_{n,i}\right )$. 
Recall that $\beta(n,i,x) -
\beta(n,i,\text{SBC}(n,i)-1)$ computes 
$\left\vert\left\{ m \in \text{I}B_{n,i}\vert \text{SBC}(n,i) \leq m \leq x\right\}\right\vert$.
It follows that $\alpha^1(n,i,x) = \alpha(n,i,x) +
\beta(n,i,\text{SBC}(n,i)-1) - \beta(n,i,x)$. 
Similarly, $\text{I}B^1_{n,i} = \text{I}B_{n,i} \setminus [\text{SBC}(n,i),\ \text{SBC}(n,i+1))$.
Thus, $\beta^1(n,i,x) = \beta (n,i,x) - \alpha (n,i,x)$ and so $\alpha^1$ and $\beta^1$
are both P-TIME.  Therefore, $\text{RI}A^1,\ \text{RI}B^1$ are both tame.

For the tameness of $\text{RI}C^1$, just note that 
$\text{I}C^1_{n,i,j} = 
\left\{m \in \text{I}B^1_{n,j}\vert\text{SBC}(n,i) \leq m < \text{SBC}(n,i+1)\right\}.$
It follows that if $x < \text{SBC}(n,i)$ then $\gamma^1(n,i,j,x) = 0$, and if
$\text{SBC}(n,i+1)-1 \leq x < 2^n$ then $\gamma^1(n,i,j,x) = 
\beta^1(n,j,\text{SBC}(n,i+1)-1) - \beta^1(n,j,\text{SBC}(n,i)-1)$.  Finally, 
if $\text{SBC}(n,i) \leq x < \text{SBC}(n,i+1) - 1$, then $\gamma^1(n,i,j,x) = 
\beta^1(n,j,x) - \beta^1(n,j,\text{SBC}(n,i)-1)$.  Thus, $\gamma^1$ is P-TIME 
and so $\text{RI}C^1$ is tame.  Note that by Lemmas 2 and 3,
we have that $\text{ERI}C^1$ is P-TIME, i.e., for all
relevant $(n,i,j)$ and all $s$ with $1 \leq s \leq \gamma^1_{n,i,j},\
c^1_{n,i,j,s}$ is a P-TIME function of $(n,i,j,s)$.

For the tameness of $\text{RI}\overline{C}^1$, note that if
$\gamma^1_{n,i,j} \leq \gamma^1_{n,j,i}$, then $\overline{\gamma}^1(n,i,j,x) = 
\gamma^1(n,i,j,x)$.  If $\gamma^1_{n,i,j} > \gamma^1_{n,j,i}$
and $i > j$, let $s = \gamma^1_{n,j,i}$.  If $x \leq c^1_{n,i,j,s}$, then 
$\overline{\gamma^1}(n,i,j,x)
= \gamma^1(n,i,j,x)$, while if $c^1_{n,i,j,s} < x < 2^n$, then
$\overline{\gamma^1}(n,i,j,x) = s$.  Finally, if $\gamma^1_{n,i,j} > 
\gamma^1_{n,j,i}$ and $i < j$, let $t = \gamma^1_{n,i,j} -
\gamma^1_{n,j,i}$.  If $x < c^1_{n,i,j,t+1}$,
then $\overline{\gamma}^1(n,i,j,x) = 0$, while if $c^1_{n,i,j,t+1} \leq
x < 2^n$, then $\overline{\gamma}^1(n,i,j,x) = \gamma^1(n,i,j,x) - t$. 
Thus, $\overline{\gamma}^1$ is P-TIME, and so
$\text{RI}\overline{C}^1$ is tame, and therefore is P-TIME
decidable.  We mention this explicitly only because it is exactly the second case condition of Equation (12).

Finally, we show that $\text{RI}A^2$ and $\text{RI}B^2$ are tame.  For $\text{RI}A^2$,
if $0 < i \leq n$, note:
\[
\text{I}A^2_{n,i} = \text{I}A^1_{n,i} \setminus \bigsqcup_{1 \leq j \leq n,\ j \neq i}\text{I}\overline{C}^1_{n,i,j} .
\]
\noindent
If $x < \text{SBC}(n,i)$, then $\alpha^2(n,i,x) = 0$; otherwise:
\[
\alpha^2(n,i,x) = \alpha^1(n,i,x) - \sum_{1 \leq j \leq n,\ j \neq i}\overline{\gamma}^1(n,i,j,x),
\]
\noindent
and so $\text{RI}A^2$ is tame.  For $\text{RI}B^2$, we have the analogous observation: for $0 < i \leq n$
\[
\text{I}B^2_{n,i} = \text{I}B^1_{n,i} \setminus  \bigsqcup_{1 \leq j \leq n,\ j \neq i}\text{I}\overline{C}^1_{n,j,i} .
\]
Then, $\displaystyle{\beta^2(n,i,x) = \beta^1(n,i,x) - \sum_{1 \leq j ,\text{IStep}(n,x),\ j \neq i}\overline{\gamma}^1(n,j,i,x)}$,
so $\text{RI}B^2$ is tame.
\end{proof}
\begin{theorem}  Each of $G,\ H$ is \textnormal{P-TIME} and is simply computed in terms of \textnormal{SBC}.
\end{theorem}
\begin{proof}
As already indicated, we do not attempt to improve on Equation (6) for
the lower bound statement for $H$.  For $G$, however, we do note
that $\text{IStep}\left ( n,2^{n-1}\right ) = \lfloor (n+1)/2\rfloor > 1$. 
Thus, $2^{n-1}$ is the largest element of $\text{I}B^1_{n,1}$, and
so if $k = \left ( G_n\right )^{-1}\left ( 2^{n-1}\right )$, then $k$ is the largest
element of $\text{I}A^1_{n,1}$.  As noted at the end of (4.1),
possibly $k \neq n$ (if $n$ is a power of 2), but, if $k \neq n$,
then $k = n-1$ (since $(1,2)$ is the only pair of consecutive 2-powers).
Thus, our analogue of Equation (8) for $G$ is:
\begin{equation}
2^n = \text{max}(G(n,n), G(n,n-1)) + \text{max}(G(n,n), G(n,n-1)).
\end{equation}
Turning to the proof that $G$ and $H$ are P-TIME, we know, by Proposition 6, that
the case conditions are P-TIME decidable, and that in
the first case we have $G(n,k) = H(n,k) = k$.  For each of
the remaining cases (case 2, for $G$ and cases 2, 3, for $H$),
we exhibit P-TIME decidable relations involving $n, k, m$ whose unique solution,
$m$, is less than $2^n$, is P-TIME and is the value of
$G(n,k)$ (resp. $H(n,k)$) as determined by the case in question.  This is
done in Equations (14), (15), (16).  

For readability, we will use $i$ as an abbreviation
for $\text{IStep}(n,k)$ in each of these equations, and in Equation (15), we will
also use $j$ as an abbreviation for $\text{Weight}(k)$, but the relations expressed
by these equations really involve only $n,k,m$.
The P-TIME computability of the unique solutions depends on Proposition 6, and on Lemma 2
(for the P-TIME computability of the relevant enumerating function) and will
be established by exhibiting an equation that specifies the computation.

For case 2 of Equation (10), the P-TIME decidable relation is:
\begin{equation}
\beta^1(n,i,m)\chi_{\text{RI}B^1}(n,i,m) =
\alpha^1(n,i,k).
\end{equation}
\noindent
The P-TIME computability of the unique solution $m = G(n,k)$ of
Equation (14) is established by:
\begin{equation*} 
\text{if }\text{RI}A^1(n,i,k)\text{ holds, then }
G(n,k) = \text{ERI}B^1(n,i,\alpha^1(n,i,k)). 
\end{equation*}

For case 2 of Equation (12), the P-TIME decidable relation is:
\begin{equation}
\overline{\gamma}^1(n,j,i,m)\chi_{\text{RI}\overline{C}^1}(n,j,i,m) =
\overline{\gamma}^1(n,i,j,k).
\end{equation}
\noindent
The P-TIME computability of the unique solution $m = H(n,k)$ of
Equation (15) is established by:
\begin{equation*} 
\text{if }\text{RI}\overline{C}^1(n,i,j,k)\text{ holds, then }
H(n,k) = \text{ERI}\overline{C}^1(n,j,i,\overline{\gamma}^1(n,i,j,k)). 
\end{equation*}

Finally, for case 3 of Equation (12), the P-TIME decidable relation is:
\begin{equation}
\beta^2(n,i,m)\chi_{\text{RI}B^2}(n,i,m) =
\alpha^2(n,\text{IStep}(n,k),k).
\end{equation}
\noindent
Note also that
\begin{equation*} 
\text{if }\text{RI}A^2(n,i,k)\text{ holds, then }
H(n,k) = \text{ERI}B^2(n,i,\alpha^2(n,i,k)). 
\end{equation*}
This means that $H(n,k)$ is the unique solution of Equation (16), and thus is P-TIME.

The argument that each of $G, H$ is simply computed in terms of SBC is similar to that in Theorem 2, and
involves the straightforward verification that all of the functions involved in Equations (14), (15), (16)
are simply computed in terms of SBC.  This traces back to examining the proof of Proposition 6.  Each
of $\alpha^1,\ \beta^1$ is explicitly and simply computed in terms of $\alpha,\beta$ and SBC.  Next,
$\gamma^1$ is explicitly and simply computed in terms of $\beta^1$ and SBC.  Then, $\overline{\gamma}^1$
is explicitly and simply computed in terms of $\gamma^1$.  Next, $\alpha^2$ is simply and explicitly computed
in terms of $\alpha^1$ and $\overline{\gamma}^1$, while $\beta^2$ is simply and explicitly
computed in terms of $\beta^1,\ \overline{\gamma}^1$ and IStep.  Finally, $\chi_{\text{RI}B^1},\ 
\chi_{\text{RI}\overline{C}^1}$ and $\chi_{\text{RI}B^2}$ can then be computed simply in terms of SBC
using the second observation in the first paragraph following the proof of Lemma 3.
\end{proof}


\begin{thebibliography}{}
\bibitem{Beame}\label{Beame}Beame, P. W., Cook, S. A., Hoover, H. J.:
Log depth circuits for division and related problems.  Siam Journal on Computing
15, 994-1003 (1986).

\bibitem{Berti1}\label{Berti1} Berti, P., Pratelli, L.,Rigo, P.: Skorohod representation 
theorem via disintegrations.  Sankhy\={a}: The Indian Journal of Statistics 
72-A, 208-220 (2010).

\bibitem{Berti2}\label{Berti2} Berti, P., Pratelli, L., Rigo, P.: A Skorohod representation 
theorem for uniform distance. Probability Theory and Related Fields150(1-2), 321-335 (2011).

\bibitem{Clote}\label{Clote}Clote, P., Kranakis, E.:  Boolean Functions and Computation Models.
Springer , Berlin-Heidelberg-New York (2002).

\bibitem{Dudley}\label{Dudley}Dudley, R.M.: Real Analysis
and Probability. Cambridge U. Press, Cambridge (1989).

\bibitem{Hesse}\label{Hesse}Hesse, W., Allender, E., Barrington, D. A. M: Uniform constant-depth
threshold circuits for division and iterated multiplication.  Journal of Computer and System
Sciences 65, 695-716 (2002).

\bibitem{Khintchine}\label{Khintchine}Khintchine, A.: \"Uber dyadische br\"uche.
Math. Zeit. 18, 109-116 (1923). 

\bibitem{Kolmogoroff}\label{Kolmogoroff}Kolmogoroff, A.: \"Uber das gesetz des 
iterierten logarithmus. Math. Ann.101, 126-135 (1929).

\bibitem{OEIS}\label{OEIS}The On-Line Encyclopedia of Integer Sequences, published electronically at http://oeis.org.

\bibitem{Papa}\label{Papa}Papadimitriou, C. H.:  Computational Complexity.
Addison-Wesley, Reading (1994). 

\bibitem{Skorokhod}\label{Skorokhod}Skorokhod, A. V.: Limit
theorems for stochastic processes.Theory of Probability and its
Applications 1, 261-290 (1956).

\bibitem{Skyers}\label{Skyers}Skyers, M.: A tale of two sequences: a 
story of convergence, weak and  almost sure. Dissertation, 
Lehigh University (2012).

\end{thebibliography}
\end{document}